\documentclass[oneside]{amsart}
\usepackage{amssymb}

\numberwithin{equation}{section}
\newcommand{\be}{\begin{equation}}
\newcommand{\ee}{\end{equation}}

\newcommand{\R}{\mathbb R}
\newcommand{\Z}{\mathbb Z}
\newcommand{\N}{\mathbb N}

\newcommand{\M}{{\mathcal M}}

\newcommand{\ep}{\varepsilon}

\renewcommand{\phi}{\varphi}
\newcommand{\pd}{\partial}

\newcommand{\co}{\colon}
\newcommand{\la}{\lambda}

\DeclareMathOperator{\grad}{grad}
\DeclareMathOperator{\trace}{trace}
\DeclareMathOperator{\vol}{vol}

\DeclareMathOperator{\diam}{diam}

\newtheorem{theorem}{Theorem}
\newtheorem{lemma}{Lemma}[section]
\newtheorem{proposition}[lemma]{Proposition}

\theoremstyle{remark}
\newtheorem{remark}[lemma]{Remark}
\newtheorem*{remark*}{Remark}

\theoremstyle{definition}
\newtheorem{definition}[lemma]{Definition}

\begin{document}

\title{A graph discretization of the Laplace-Beltrami operator}

\author{Dmitri Burago}                                                          
\address{Dmitri Burago: Pennsylvania State University,                          
Department of Mathematics, University Park, PA 16802, USA}                      
\email{burago@math.psu.edu}                                                     
                                                                                
\author{Sergei Ivanov}
\address{Sergei Ivanov:
St.Petersburg Department of Steklov Mathematical Institute,
Russian Academy of Sciences,
Fontanka 27, St.Petersburg 191023, Russia}
\email{svivanov@pdmi.ras.ru}

\author{Yaroslav Kurylev}                                                          
\address{Yaroslav Kurylev:
Department of Mathematics, University College London, Gower Street,
London, WC1E 6BT, UK}                      
\email{y.kurylev@math.ucl.ac.uk}        

\thanks{The first author was partially supported
by NSF grant DMS-1205597.
The second author was partially supported by
RFBR grant 11-01-00302-a.
The third author was partially supported by 
EPSRC}



\begin{abstract}
We show that eigenvalues and eigenfunctions of the Laplace--Beltrami operator
on a Riemannian manifold are approximated by eigenvalues
and eigenvectors of a (suitably weighted) graph Laplace operator of
a proximity graph on an epsilon-net.
\end{abstract}

\maketitle

\section{Introduction}
 
Traditionally, discretization in Riemannian geometry
is associated with triangulations and other polyhedral approximations.
This approach works perfectly well in dimension two
but meets a number of obstacles in higher dimensions.
It is now clear, due to works of Cheeger, Petrunin, Panov and others
(see \cite{Cheeger,Petrunin,Panov,Orshanskiy})
that in dimensions beyond three
polyhedral structures are too rigid to serve as discrete models
of Riemannian spaces with curvature bounds.
In some applications, we get a Riemannian manifold as
a cloud of points with approximate distances between them,
see e.g.\ \cite{KKL}, \cite{Belkin}.
It appears that point clouds arising as discretizations of Riemannian
manifolds can be effectively distinguished from arbitrary ones.
This issue will be addressed elsewhere.
For triangulations, even the problem of determining whether a given simplicial complex
is a topological manifold is algorithmically undecidable
(see e.g.\ \cite[\S6.2]{Poonen} for a simple proof).

In a few papers, we will try to discuss approximating Riemannian manifolds by graphs,
of course with additional structures and various boundedness conditions.
Here we show that the spectrum of a suitable graph
Laplacian gives a reasonable approximation
to the spectrum of the Riemannian Laplace--Beltrami operator.  
The key difference with finite element and similar methods 
(see e.g.\ \cite{Ciarlet}, \cite{DP} and an interesting recent work \cite{Aubry})
is that in our
construction the set of vertices is an arbitrary net as long as it is dense enough.  There are no 
local regularity constraints and we use only very rough data.

Let us note that we look at the problem from the viewpoint of spectral (Riemannian) geometry. 
On the other hand, similar problems of course have been receiving a lot of attention from 
numerical analysts. The most closely related formulations can be found in the above mentioned 
\cite{Belkin}, however it gives only a probabilistic result with no constructive suggestion
of how one decides which point clouds do the job. An ideologically close (but still rather different) 
approach can be found in \cite{BBJ} and references therein. 

We do not discuss numerical and computational aspects of our results.
In the level of justification, our proofs seem to be relatively technical. Still, it seems that practical implementation of 
computational methods behind our theorems should be a relatively easy task. We do not address 
this issue here but hope to do this elsewhere. Let us just mention that we start with an arbitrary approximation
of our Riemannian manifold by a finite metric-measure space. Then we associate to this approximation 
a (sparse) matrix in the most straightforward way. In particular, in 
Section \ref{sec:volume} we describe some way of assigning to a given $\ep$-net on a 
Riemannian manifold a 
proper graph approximation. Once the matrix is constructed,  its
eigenvalues turn out to be very good approximations to those of the Riemannian Laplacian.

Let $M^n$ be a compact Riemannian manifold (without boundary)
and $X\subset M$ a finite $\ep$-net.
The geodesic distance between $x,y\in M$ is denoted by $d(x,y)$
or simply $|xy|$.
Given such $X$ and $\rho>0$,
one constructs a \textit{proximity graph}
$\Gamma=\Gamma(X,\rho)$: the set of vertices of the graph is $X$,
and two vertices are connected by an edge if and only if $d(x,y)<\rho$.
In our set-up, we assume that $\ep\ll\rho$ and $\rho$ is sufficiently small
so that $\rho$-balls in $M$ are (bi-Lipschitz) close to Euclidean.
In addition, we assign weights to vertices and edges of $\Gamma$
as explained below.
Then there is a graph Laplacian operator associated with this structure,
see \eqref{26.7.7}.
Our goal is to approximate the eigenvalues $\la_k(M)$
the Laplace--Beltrami operator on~$M$
by eigenvalues $\la_k(\Gamma)$ of the
graph Laplacian.

This kind of problems were studied before.
Fujiwara \cite{Fu} showed that, if $X$
is an $\ep$-separated $\ep$-net and $\rho=5\ep$,
then the eigenvalues of (unweighted) graph Laplacian
of the proximity graph after proper normalization satisfy
$$
 C_n^{-1} \la_k(M) \le \la_k(\Gamma) \le C_n \la_k(M)
$$
where $C_n>0$ is a constant depending only on $n=\dim M$.
Belkin and Niyogi \cite{Belkin} considered random,
uniformly distributed nets in $M$ and showed that,
for a suitable choice of edge weights (depending on distances),
the spectrum of the resulting graph Laplacian converges
to the spectrum of~$M$ in the probability sense.

In this paper we present a construction that works for an
arbitrary net. The ``density'' of the net may vary 
from one region to another. To compensate for this,
we need to introduce weights on vertices.
These weights determine a discrete measure
on $X$ and we essentially
require that $X$ approximates $M$ as a metric measure space.

\subsection*{The construction}
Let $\ep>0$ and $X=\{x_i\}_{i=1}^N$ be a finite $\ep$-net in $M$.
We denote by $B_r(x)$ the closed metric ball of radius $r$ centered at $x\in M$.
We assume that $X$ is equipped with
a discrete measure $\mu=\sum\mu_i\delta_{x_i}$
which approximates the volume of $M$ in the following sense.

\begin{definition}
\label{d:vol-approx}
A measure $\mu$ on $X$ is an \textit{$\ep$-approximation} of volume $\vol$ on $M$
if there exist a partition of $M$ into measurable subsets $V_i$, $i=1,\dots,N$,
such that $V_i\subset B_\ep(x_i)$ and $\vol(V_i)=\mu_i$ for every $i$.

In this case we also say that the pair $(X,\mu)$ \textit{$\ep$-approximates}
$(M,\vol)$.
\end{definition}

Every $\ep$-net $X$ in $M$ can be equipped with such a measure.
For example, let $\{V_i\}$ be the Voronoi decomposition of $M$
with respect to $X$ and $\mu_i=\vol(V_i)$.
We discuss other constructions and some properties
of Definition~\ref{d:vol-approx} in Section~\ref{sec:volume}.
In particular, we show that this definition is naturally related
to weak convergence of measures (see Remark \ref{r:prokhorov}).

Consider the space $L^2(X)=L^2(X,\mu)$, that is
the $N$-dimensional space of functions from $X$ to $\R$
equipped with the following inner product:
\be \label{31.7.3}
 \langle u,v\rangle=\langle u,v\rangle_{L^2(X)} = \sum \mu_i u(x_i)v(x_i) ,
\ee
or, equivalently, with a Euclidean norm given by
\be\label{e:L2Xnorm}
 \|u\|^2 = \|u\|^2_{L^2(X)} = \sum \mu_i |u(x_i)|^2 .
\ee
We think of $L^2(X)$ as a finite-dimensional approximation
to $L^2(M)$. For the sake of brevity,
we omit the index $L^2(X)$ in most formulae 
in the paper.

We define the following weighted graph $\Gamma=\Gamma(X,\mu,\rho)$.
The set of vertices is $X$, two vertices 
$x,y\in X$ are connected by an edge if and only if $d(x,y)<\rho$.
We write $x\sim y$ for $x,y\in X$ if
they are connected by an edge. 
Both vertices and edges are weighted.
The weight of a vertex $x_i$ is $\mu_i$.
To an edge $e_{ij}=(x_i,x_j)$ we associate a weight
$w(e_{ij})=w_{ij}$ given by
$$
 w_{ij} = \frac{2(n+2)}{\nu_n\rho^{n+2}} \mu_i\mu_j 
$$
where $\nu_n$ is the volume of the unit ball in $\R^n$.
Note that $w_{ij}=w_{ji}$.

We approximate the Riemannian Laplace--Beltrami operator $\Delta=\Delta_M$
by the weighted graph Laplacian $\Delta_\Gamma:L^2(X)\to L^2(X)$
defined by
\be \label{26.7.7}
\begin{aligned}
 (\Delta_\Gamma u)(x_i) &= \frac1{\mu_i} \sum_{j:x_j\sim x_i} w_{ij} (u(x_j)-u(x_i)) \\
 &=\frac{2(n+2)}{\nu_n\rho^{n+2}}\sum_{j:x_j\sim x_i} \mu_j (u(x_j)-u(x_i)) .
\end{aligned}
\ee
The motivation behind this formula is the following.
If $u$ is a discretization of a smooth function $f\co M\to\R$, then the latter sum is
the discretization of an integral over the ball $B_{\rho}(x_i)$:
$$
\sum_{j:x_j\sim x_i} \mu_j (u(x_j)-u(x_i)) \approx \int_{B_\rho(x_i)} (f(x)-f(x_i)) \,dx ,
$$
and the normalization constant is chosen so that the normalized
integral approaches $\Delta f(x_i)$ as $\rho\to 0$, see Section \ref{sec:integral}.
It follows that the graph Laplacian of the discretization of $f$
approximates $\Delta f$ if $\ep\ll\rho\ll 1$.

\begin{remark}
\label{r:edge-weights}
One can introduce weights on edges depending on their lengths.
For example, the above value of $w_{ij}$ could be multiplied by
$\varphi(\rho^{-1}d(x_i,x_j))$
where $\varphi$ is a nonnegative non-increasing function on $[0,1]$.
With a suitably adjusted normalization
constant, everything generalizes to this set-up in a straightforward way.
Probably a smart choice of $\varphi$ can allow one to improve the rates of 
convergence.
\end{remark}

The operator $\Delta_\Gamma$ is self-adjoint with respect to the
inner product \eqref{31.7.3} on $L^2(X)$ and nonpositive definite,
see Section \ref{sec:discrete-differential}.
Let $0=\la_1(\Gamma)\le\la_2(\Gamma)\le\dots\le\la_N(\Gamma)$
be the eigenvalues of $-\Delta_\Gamma$ and
$0=\la_1(M)\le\la_2(M)\le\dots$ the eigenvalues of $-\Delta_M$.

\subsection*{Statement of the results}
Let $\M=\M_n(K, D, i_0)$ be the class of $n$-dimensional Riemannian manifolds 
with absolute value of sectional curvature bounded by $K$, 
diameter bounded by $D$ and injectivity
radius bounded below by $i_0$.

Throughout the paper, we denote by $C$ various absolute constants
whose precise value may vary from one occurrence to another (even within one formula).
We write $C_n$, $C_\M$, etc, to denote constants depending only on 
the respective parameters.

In some of the arguments we denote by $\sigma,\sigma_1,\dots$,
various ``small'' quantities depending on $\ep$, $\rho$, etc.
These notations are local and redefined in each proof
where they are used.

The main result of the paper is the following

\begin{theorem} \label{main}
For every integer $n\ge 1$ there exist positive constants $C_n$ and $c_n$
such that the following holds.
Let $M\in \M=\M_n(K, D, i_0)$ and 
$\Gamma=\Gamma(X,\mu,\rho)$ be a weighted graph defined as above,
where $(X,\mu)$ $\ep$-approximates $(M,\vol)$,
$\rho<i_0/2$, $K\rho^2<c_n$ and
$\ep/\rho<\min\{1/n,1/3\}$.

Then for every $k\in\Z_+$
such that $\rho\la_k(M)<c_n$
one has
\be \label{main_th}
|\la_k(\Gamma)-\la_k(M)| \le C_n (\ep/\rho+K\rho^2)\la_k(M)+C_n\rho\la_k(M)^{3/2} .
\ee
Therefore
$$
 |\la_k(\Gamma)-\la_k(M)| \le C_{\M,k}(\ep/\rho+\rho)
$$
provided that $\rho<C_{\M,k}^{-1}$.

As a corollary, for every fixed $k$ we have $\la_k(\Gamma)\to\la_k(M)$
as $\rho\to 0$ and $\frac\ep\rho\to 0$ and the convergence is
uniform over all $M\in\M$.
\end{theorem}

The estimate \eqref{main_th} is a combination of Proposition \ref{p:upper-bound}
and Proposition \ref{p:lower-bound} where we prove
an upper and a lower bound, respectively,
for $\la_k(\Gamma)$ in terms of $\lambda_k(M)$.
These propositions also provide somewhat sharper estimates
on $\la_k(\Gamma)-\la_k(M)$.
The second assertion of Theorem \ref{main} follows from \eqref{main_th}
and the fact that for every fixed $k$ the eigenvalue $\la_k(M)$
is uniformly bounded over $M\in\M$.

Our next result establishes convergence of eigenfunctions.
Namely, it is possible to approximate an eigenfunction of $\Delta_M$ 
corresponding to an eigenvalue $\la$ by a 
linear combination of eigenfunctions of $\Delta_\Gamma$
corresponding to eigenvalues close to~$\la$.
The precise formulations are given in Theorems \ref{t:eigenspace}
and~\ref{t:eigenfunction} in Section \ref{sec:eigenfunction}.
Here we give only a special case of this result
where $\lambda$ has multiplicity~$1$.

\begin{theorem} \label{eigenfunction_1}
Let $f_k$ be a unit-norm eigenfunction of $-\Delta_M$
corresponding to an eigenvalue $\la_k=\la_k(M)$ of multiplicity~$1$,
and let $\delta_\la=\min\{1,\la_{k+1}-\la_k,\la_k-\la_{k-1}\}$.
Then, for sufficiently small $\rho$ and $\ep/\rho$
(more precisely, if $\rho+\ep/\rho< C_{\M,k}^{-1}\delta_\la$),
the eigenvalue $\la_k(\Gamma)$ of $-\Delta_\Gamma$
also has multiplicity~1, and for a corresponding
unit-norm eigenvector $u_k$ we have
$$
\begin{aligned}
\|P f_k-   u_k\|_{L^2(X)}
&\le C_{\M,k}\delta_\la^{-1}(\ep/\rho+\rho),
\\
\|Iu_k-  f_k\|_{L^2(M)}
&\le C_{\M,k}\delta_\la^{-1}(\ep/\rho+\rho),
\end{aligned}
$$
where the norm $\|\cdot\|_{L^2(X)}$ is defined in \eqref{e:L2Xnorm}
and the maps $P\co L^2(M)\to L^2(X)$ and $I\co L^2(X)\to C^{0,1}(M)$
are discretization and  interpolation 
defined in Definitions \ref{P} and~\ref{I},
respectively.
\end{theorem}

Theorem \ref{eigenfunction_1} is a special case of Theorem \ref{t:eigenfunction},
which handles arbitrary multiplicity.
Theorem \ref{t:eigenspace} is another variant where an estimate is uniform
over $\M$ (in particular, it does not depend on the size of spectral gaps).
However the rate of convergence guaranteed by Theorem~\ref{t:eigenspace} is
not as good as in Theorem~\ref{t:eigenfunction}.

\subsection*{Remarks on the proof}
Let us note that the upper bound
\be\label{e:upper-bound-imprecise}
\limsup \la_k(\Gamma)\le \la_k(M)
\ee
is nearly trivial.
It follows from the fact that our graph Laplacian
approximates the function Laplacian for every smooth function.
Indeed, let $f_1,\dots,f_k$ be orthonormal eigenfunctions
of $-\Delta_M$ with eigenvalues $\la_1(M),\dots,\la_k(M)$.
It is well known that the eigenfunctions are smooth
(more precisely, their $C^3_*$ norms are bounded by
$C_{\M,k}$, see \cite{AKKLT}).
Let $u_1,\dots,u_k\in L^2(X)$ be discretizations of $f_1,\dots,f_k$.
(For smooth functions the precise definition of discretization
does not really matter; one can define e.g.\ $u_j(x_i)=f_j(x_i)$.)
Since the functions $f_j$ are smooth, their discrete functions $u_j$
associated to them are almost orthonormal in $L^2(X)$ and their discrete Laplacians
$\Delta_\Gamma(u_j)$ are pointwise close to the Laplacians $\Delta_M f_j$.
Hence
$$
 \langle -\Delta_\Gamma u_j,u_j \rangle
 \approx \langle -\Delta_M f_j,f_j \rangle_{L^2} = \la_j(M)
$$
and therefore
$
 \langle -\Delta_\Gamma u,u \rangle \lesssim \la_k(M)
$
for every $u$ from the linear span of $u_1,\dots,u_k$.
Thus we have a $k$-dimensional subspace of $L^2(X)$ where
the norm of the discrete Diriclet energy functional \eqref{e:energy}
is bounded by approximately $\la_k(M)$.
By the minimax principle it follows that $\la_k(\Gamma)\lesssim \la_k(M)$,
in other words, \eqref{e:upper-bound-imprecise} holds.

The proof of the upper bound  in Sections \ref{sec:estimates}
and \ref{sec:discretization} is different.
We define a discretization map $P\co L^2(M)\to L^2(X)$
that makes sense for non-smooth functions
and show that this map almost preserves the $L^2$ norm
and almost does not increase the Diriclet energy,
on a bounded energy level (see Definition \ref{P} and Lemma \ref{l:P-properties}).
This argument does not require pointwise eigenfunction estimates
and yields sharper inequalities.

The lower bound (i.e., the inequality $\liminf \la_k(\Gamma)\ge \la_k(M)$)
is more delicate.
Here good approximation of Laplacians of smooth functions is not sufficient.
For example, consider a disjoint union $\Gamma$ of two graphs $\Gamma_1$
and $\Gamma_2$ each of
which provides a good approximation of the function Laplacian.
The graph Laplacian $\Delta_\Gamma$ approximates the function Laplacian
as well as $\Delta_{\Gamma_1}$ and $\Delta_{\Gamma_2}$ do,
but the spectrum is different: every eigenvalue appears twice.

To prove the upper bound, we construct
a map $I\co L^2(X)\to C^{0,1}(M)$,
called the interpolation map,
with properties similar to those of the discretization map~$P$,
see Definition~\ref{I} and Lemma~\ref{l:interpolation}.
This map is essentially a convolution with a certain kernel
(the form of the kernel is essential for the estimate
in Lemma~\ref{l:interpolation}(2)).
With this map, the proof of the lower bound is similar
to that of the upper bound.

In addition, the maps $P$ and $I$ are almost inverse
to each other on bounded energy levels (Lemma \ref{l:inverse}).
These properties of $P$ and $I$ imply our eigenfunction estimates
(Theorems  \ref{eigenfunction_1}, \ref{t:eigenspace} and~\ref{t:eigenfunction})
by means of linear algebra arguments.

\begin{remark}
\label{r:stability}
The input data to the construction are $\rho>0$
and the finite metric measure space $(X,\mu)$.
One naturally asks how sensitive are the resulting
eigenvalues $\la_k(\Gamma)$ to ``measurement error''
in these data.
A small relative error in weights $\mu_i$ results
in a relative error of the same order in the $L^2(X)$
and the discrete Dirichlet energy \eqref{e:energy}
and hence to the eigenvalues.
A small (of order $\ep$) variation of distances in $X$
changes the set of edges of $\Gamma$: some edges of lengths
$\rho\pm\ep$ may be added or removed.
The discrete Dirichlet energy and hence the eigenvalues
are clearly monotone with respect to adding edges.
Therefore the eigenvalues are bounded above by those
of the proximity graph defined by the parameter $\rho+\ep$
in place of $\rho$, up to a factor $(1+\ep/\rho)^{n+2}$.
A similar argument yields a lower bound.
\end{remark}

\begin{remark}
The convergence of eigenvalues in Theorem \ref{main} is
uniform on a larger class of $n$-manifolds, namely those with
bounded Ricci curvature and diameter and injectivity radius separated from zero.
Indeed, by \cite{Anderson} this class is pre-compact in $C^{1,\alpha}$
(and hence Lipschitz) topology. This pre-compactness and convergence for
every individual manifold implies uniform convergence on the class.
This can be shown with an argument similar to one outlined
in Remark \ref{r:stability}.
\end{remark}

\begin{remark}
If the weights $\mu_i$ are constant
(i.e., $\mu_i=\mu_0:=\vol(M)/N$),
then the edge weights in our construction are also constant.
Hence the graph Laplacian given by \eqref{26.7.7} is the ordinary
(unweighted) graph Laplacian multiplied by a constant. Also note that in this case
the degree in the graph is almost constant (up to a small relative error):
the degree of every vertex approximately equals $\nu_n\rho^n/\mu_0$.

Unweighted graph Laplacians has been studied much more thoroughly
than weighted ones.
If necessary, one can make the weights constant (at the expense
of increasing the number of vertices)
as follows.
First approximate the weights $\mu_i$ by rational multiples of $\vol(M)$
and let $q$ be a common denominator of these rationals.
Then replace every point $x_i$ with weight $\mu_i=\frac{p_i}q \vol(M)$ by $p_i$
points (at almost the same location) with weights equal to $\vol(M)/q$.
The resulting metric measure space approximates $(M,\vol)$
as well as the original one do.
\end{remark}

\begin{remark}
Although our point is to avoid triangulation of a manifold,
let us mention that triangulation-based techniques
allow one to handle differential form Laplacians as well, see \cite{DP}.
It is interesting whether a suitable generalization of a graph Laplacian
can be used for this purpose too.
One can show that the spectrum of the differential form Laplacian
is continuous with respect to Gromov--Hausdorff topology on $\M$.
Hence a Gromov--Hausdorff approximation of a manifold
(such as an $\ep$-net) determines differential form Laplacian eigenvalues
up to a small error. However an explicit procedure of such determination
is yet to be found.
\end{remark}

\subsection*{Organization of the paper}
In Section \ref{sec:prelim} we collect various preliminaries.
In Section \ref{sec:estimates} we prove some technical results
about average dispersion  in $r$-balls of a function $f\in L^2(M)$.
This quantity, denoted by $E_r(f)$, is used throughout the paper
as an intermediate step between Dirichlet energy in $H^1(M)$
and its discretization.
In Section \ref{sec:discretization} we define the discretization map $P$
and prove an upper bound for the graph eigenvalues
(Proposition~\ref{p:upper-bound}).
Section~\ref{sec:smoothening} is devoted to properties of
a smoothening operator (the convolution with a special kernel)
used in the definition of the interpolation map~$I$.
The key result there is Lemma~\ref{l:dLambdaf}.
In Section~\ref{sec:interpol} we define $I$ and prove
a lower bound for the graph eigenvalues (Proposition~\ref{p:lower-bound}).
Proofs of the main results are contained in Section~\ref{sec:eigenfunction}.
In Section~\ref{sec:volume} (which is formally independent of the rest of the paper)
we discuss various aspects of volume approximation in the sense
of Definition~\ref{d:vol-approx}.

\medskip
\textit{Acknowledgement}.
This work began at Newton Institute, Cambridge,
where the second and third named authors
met at ``Inverse Problems'' scientific programme in September, 2011.
It was continued at Fields Institute, Toronto,
during ``Geometry in Inverse Problems'' program 
in April, 2012.
We are grateful to these institutions for
wonderful research environment they provided.

We are grateful to Yashar Memarian and Christian B\"ar
for bringing our attention
to some of the references,
and Dmitry Chelkak, Alexander Gaifullin, Alexander Nazarov, Mark Sapir
for fruitful discussions and helpful remarks.

\section{Preliminaries}
\label{sec:prelim}

\subsection{Discrete differential}
\label{sec:discrete-differential}
Let $E=E(\Gamma)$ be the set of directed edges of our graph.
(Each pair of adjacent vertices gives rise to two elements of $E$.)
Recall that every edge $e_{ij}=(x_i,x_j)$
is equipped with a weight $w(e_{ij})=w_{ij}$.
By $L^2(E)$ we denote the space
of real-valued functions on $E$ equipped with the following
inner product:
$$
 \langle \xi,\eta\rangle_{L^2(E)} = \frac12\sum_{e\in E} w(e) \xi(e)\eta(e).
$$
For a discrete function $u\colon X\to\R$ we define its
\emph{discrete differential} $\delta u\colon E\to\R$ by 
\begin{equation}
\label{by-parts}
 (\delta u)(e_{ij}) = u(x_j) - u(x_i) .
\end{equation}
The \emph{discrete Dirichlet energy} functional of $\Gamma$ is the quadratic
form
\be \label{31.7.2}
 u \mapsto \|\delta u\|_{L^2(E)}^2 = \langle \delta u,\delta u\rangle_{L^2(E)}
\ee
on $L^2(X)$.
A straightforward calculation shows that
\be \label{31.7.6}
 \langle \Delta_\Gamma u, v\rangle_{L^2(X)} = -\langle \delta u, \delta v\rangle_{L^2(E)},
\ee
in particular, $\langle \Delta_\Gamma u, v\rangle = \langle u, \Delta_\Gamma v\rangle$
and $\langle \Delta_\Gamma u, u\rangle = -\|\delta u\|^2$
for all $u,v\in L^2(X)$.
(Here and almost everywhere in the paper we omit indices $L^2(X)$ and $L^2(E)$.)
Thus $\Delta_\Gamma$ is self-adjoint and nonpositive on $L^2(X)$.

The above consideration does not depend on a particular choice of weights.
In our case, the discrete Dirichlet energy $\|\delta u\|=\|\delta u\|_{L^2(E)}$ is given by
\begin{equation}
\label{e:energy}
 \|\delta u\|^2
 = \frac{n+2}{\nu_n\rho^{n+2}} \sum_{i,j:x_j\sim x_i} \mu_i\mu_j |u(x_i)-u(x_j)|^2 .
\end{equation}

Since the operator $-\Delta_\Gamma$ is self-adjoint on $L^2(X)$
and the associated quadratic form is the discrete Dirichlet energy,
the minimax principle applies:
$$
 \la_k(\Gamma) = \min_L \max_{u\in L\setminus 0} \frac {\|\delta u\|^2_{L^2(E)}} {\|u\|^2_{L^2(X)}}
$$
where the minimum is taken over all $k$-dimensional subspaces $L\subset L^2(X)$.

\subsection{Local Riemannian geometry}
\label{sec:riem-prelim}
Throughout the paper, $M$ is a compact Riemannian
manifold (without boundary) and $n=\dim M$.
The absolute values of
sectional curvatures of $M$ are  bounded above by $K$
and the injectivity radius is bounded below by $i_0$.
Our standing assumptions are that $\rho<i_0/2$ and $K\rho^2<1/n^2$.

For $x\in M$, $\exp_x\colon T_xM\to M$ is the Riemannian exponential map.
We always restrict $\exp_x$ to the ball $B_{2\rho}(0)\subset T_xM$,
this restriction is a diffeomorphism onto the geodesic ball $B_{2\rho}(x)$
and hence its inverse $\exp_x^{-1}\colon B_{2\rho}(x)\to B_{2\rho}(0)$
is well-defined. We denote the Jacobian of $\exp_x$ at $v\in B_{2\rho}(0)\subset T_xM$
by $J_x(v)$.

By the Rauch Comparison Theorem, the relative distortion of metric by $\exp_x$
at $v\in B_{2\rho}(0)\subset T_xM$ is bounded by $O(K|v|^2)$ and hence
\be\label{e:rauch-jacobian}
  (1+CnK|v|^2)^{-1}\le J_x(v) \le 1+CnK|v|^2 .
\ee
It follows that
$\vol(B_r(x))\sim \nu_n r^n$ as $r\to 0$, more precisely,
$$
 \left|\vol(B_r(x))-\nu_n r^n\right| \le CnKr^{n+2}
$$
for all $r<2\rho$.

The inner product in $T_xM$ defined by the Riemannian structure
is denoted by~$\langle,\rangle$. 
This scalar product allows one to identify $T_xM$ and $T_x^*M$
and we sometimes assume this identification to simplify notation.
By $\grad f(x)$ we denote the Riemannian gradient
of a function $f\colon M\to\R$ at $x\in M$,
i.e., the vector in $T_xM$ corresponding to
the differential $d_xf\in T^*M$.
Recall that the gradient of the distance function $d(\cdot,y)$
at $x$ is the velocity vector at the endpoint
of the minimal geodesic from $y$ to $x$, that is,
\be
\label{e:derivative-of-distance}
 \grad d(\cdot,y)(x) = -\frac{\exp_x^{-1}(y)}{d(x,y)} .
\ee

\subsection{Integration over balls}
\label{sec:integral}

In this section we justify the normalization constant in \eqref{26.7.7}.
If $Q$ is a quadratic form on $\R^n$, then for every $r>0$ we have
\be\label{e:integralQ}
 \int_{B_r(0)} Q(x)\,dx = \frac{\nu_n r^{n+2}}{n+2} \trace(Q) .
\ee
Indeed, since both sides are preserved under orthogonal transformations
and linear in $Q$, one can replace $Q$ by its average
under the action of the orthogonal group.
Thus it suffices to verify \eqref{e:integralQ} only for rotation-invariant
quadratic forms, or, equivalently, for the form $Q(x)=|x|^2$.
For this form, one computes the integral using spherical coordinates:
$$
 \int_{B_r(0)} |x|^2\,dx = \int_0^r t^2 \vol_{n-1}(\pd B_t(0))\,dt = 
 \int_0^r n\nu_n t^{n+1}\,dt = \frac{n\nu_n r^{n+2}}{n+2} .
$$
The identity \eqref{e:integralQ} follows since $\trace(x\mapsto|x|^2)=n$.

Let $f\co\R^n\to\R$ be a smooth function. Integrating the Taylor expansion
of $f$ at $x_0\in\R^n$,
$$
 f(x)-f(x_0) = L(x-x_0)+Q(x-x_0) + o(|x-x_0|^2), \qquad |x-x_0|\to 0,
$$
where $L=d_{x_0}f$ and $Q=\frac12 d^2_{x_0}f$, using \eqref{e:integralQ}, yields
$$
 \int_{B_r(x_0)} (f(x)-f(x_0))\,dx 
 =\frac{\nu_n r^{n+2}}{n+2} \trace(Q) + o(r^{n+2}) \\
 =\frac{\nu_n r^{n+2}}{2(n+2)} \Delta f(x_0) + o(r^{n+2})
$$
as $r\to 0$.
For a smooth function $f\co M\to\R$ this relation holds as well since
the Jacobian of the exponential
map introduces an error term of order $O(r^{n+3})$,
as follows easily from \eqref{e:rauch-jacobian}.
Thus
$$
 \frac{2(n+2)}{\nu_n \rho^{n+2}}\int_{B_\rho(x_0)} (f(x)-f(x_0))\,dx \to \Delta f(x_0)
 \qquad\text{as $\rho\to 0$}
$$
for every smooth $f\co M\to\R$ and every $x_0\in M$.
Furthermore the error term is controlled by the
modulus of continuity of the second derivative of~$f$.

Replacing the above integral by the sum from \eqref{26.7.7}
essentially replaces the integration over the ball
by integration over the union of the sets $V_i$
(see Definition~\ref{d:vol-approx}) such that the
respective points $x_i$ belong to the ball.
One easily sees that the error term
introduced by this change is controlled by $\ep/\rho^2$.
(This estimate can be improved by introducing edge weights
as in Remark~\ref{r:edge-weights}).
It follows that the discrete Laplacian of a smooth function
approaches its ordinary Laplacian as $\rho+\ep/\rho^2\to 0$.

This observation is important for motivation of our
definitions, but we do not use it in the proofs.
Our arguments are based on the discrete Dirichlet energy
and the minimax principle which provide better estimates.

\section{Some estimates}
\label{sec:estimates}

In this section we prove some inequalities
for functions on $M$ not involving discretization.

\begin{definition}
Let $f\in L^2(M)$ and $0<r<2\rho$.
For every measurable set $V\subset M$, define
$E_r(f,V)\in\R_+$ by
$$
 E_r(f,V) = \int_V \int_{B_r(x)} |f(y)-f(x)|^2 \,dy\,dx .
$$
Let $E_r(f)=E_r(f,M)$.
\end{definition}

\begin{remark}
The quantity $E_r(f)$ is bounded in terms of $\|f\|_{L^2}$,
namely 
\be\label{e:Er-bounded}
E_r(f) \le C\nu_n r^n \|f\|^2_{L^2} .
\ee
Indeed,
$$
\begin{aligned}
 E_r(f) &\le  2 \int_M \int_{B_r(x)} (|f(x)|^2+|f(y)|^2) \,dy\,dx \\
 &=  4 \int_M \int_{B_r(x)} |f(x)|^2 \,dy\,dx = 4\int_M \vol(B_r(x)) |f(x)|^2 \,dx ,
\end{aligned}
$$
and the right-hand side is bounded above by $C\nu_n r^n \|f\|^2_{L^2}$.
Since $E_r$ is a nonnegative quadratic form, \eqref{e:Er-bounded}
implies that it is a continuous map from $L^2(M)$ to~$\R_+$.
\end{remark}

\begin{lemma}\label{l:Er-df-L2}
Let $f\in H^1(M)$ and $0<r<2\rho$. Then
$$
 E_r(f) \le \big(1+CnKr^2\big) \frac{\nu_n}{n+2} r^{n+2} \|df\|_{L^2}^2 .
$$
\end{lemma}

\begin{remark*}
The inequality turns to almost equality if $f$ is smooth
and $r$ is small.
This follows from the fact that the constant $\frac{\nu_n}{n+2}$ is the
integral of the square of a coordinate function
over the unit ball in $\R^n$, see \eqref{e:integralQ}.
\end{remark*}

\begin{proof}[Proof of Lemma \ref{l:Er-df-L2}]
Since smooth functions are dense in $H^1(M)$ and
$E_r$ is a continuous map from $H^1(M)$ to $\R_+$, we may assume
that $f$ is smooth. Thus we can speak about pointwise values
and derivatives of~$f$.

For every $x\in M$, we have
$$
 \int_{B_r(x)} |f(y)-f(x)|^2\,dy
 = \int_{B_r(0)\subset T_xM} |f(\exp_x(v))-f(x)|^2 J_x(v)\,dv
$$
where $J_x$ is the Jacobian of $\exp_x$
(see Section~\ref{sec:riem-prelim})
and $\int dv$ denotes the integration with respect
to the Euclidean volume on $T_xM$ determined by
the Riemannian scalar product.
Since $J_x(v)\le 1+CnKr^2$ for all $v\in B_r(0)\subset T_xM$, it suffices to prove that
\be\label{e:Er-df-L2-1}
 A:=\int_M\int_{B_r(0)\subset T_xM} |f(\exp_x(v))-f(x)|^2\,dv dx
 \le \frac{\nu_n}{n+2} r^{n+2} \|df\|_{L^2}^2 .
\ee
For every $x$ and $v$ we have
$$
 f(\exp_x(v))-f(x) = \int_0^1 \tfrac d{dt} f(\exp_x(tv)) \,dt
 =\int_0^1 df(\Phi_t(x,v)) \,dt
$$
where $\Phi_t\co TM\to TM$ is the time $t$ geodesic flow,
namely $\Phi_t(x,v)=(\gamma_{x,v}(t),\gamma_{x,v}'(t))$
where $\gamma_{x,v}$ is the constant-speed geodesic
given by $\gamma_{x,v}(t)=\exp_x(tv)$.
In the expression $df(\Phi_t(x,v))$, the derivative $df$
is regarded as a (fiberwise linear) map from $TM$ to~$\R$.

The above identity and the Cauchy--Schwartz inequality imply that
$$
 |f(\exp_x(v))-f(x)|^2 \le \int_0^1 |df(\Phi_t(x,v))|^2 \,dt .
$$
Hence the right-hand side of \eqref{e:Er-df-L2-1} can be estimated as follows:
\be\label{e:Er-df-L2-2}
 A\le \int_0^1 \int_{\mathcal B(r)} |df(\Phi_t(\xi))|^2 \,d\vol_{TM}(\xi)\, dt
\ee
where $\mathcal B(r)\subset TM$ is the set of all tangent vectors
$\xi\in TM$ such that $|\xi|\le r$, and $\vol_{TM}$ is the standard $2n$-dimensional
volume form on $TM$. Since $\mathcal B(r)$ is invariant under $\Phi_t$
and $\Phi_t$ preserves $\vol_{TM}$ (by Liouville's Theorem), the inner integral
in \eqref{e:Er-df-L2-2} does not depend on~$t$. Therefore
$$
\begin{aligned}
 A &\le \int_{\mathcal B(r)} |df(\xi)|^2 \,d\vol_{TM}(\xi)
 =\int_M \int_{B_r(0)\subset T_xM} |d_xf(v)|^2 \,dvdx \\
 &=\int_M \frac{\nu_n}{n+2} r^{n+2} |d_xf|^2 \,dx
 =\frac{\nu_n}{n+2} r^{n+2} \|df\|_{L^2}^2
\end{aligned}
$$
where the second identity follows from \eqref{e:integralQ}.
This proves \eqref{e:Er-df-L2-1} and hence the lemma.
\end{proof}

\begin{lemma}
\label{l:dispersion}
Let $0<r<2\rho$, $f\in L^2(M)$ and $V\subset M$ be a measurable set such that
$\vol(V)=\mu>0$ and
$\diam(V)\le 2\ep$ where $\ep<r$.
Let
$
a = \mu^{-1} \int_V f(x)\,dx
$
be the integral mean of $f|_V$.
Then
$$
 \int_V |f(x)-a|^2\,dx \le \frac C{\nu_n(r-\ep)^n} E_r(f,V) .
$$
\end{lemma}

\begin{proof}
A standard computation shows that
\be\label{e:dispersion1}
 \int_{V} |f(x)-a|^2\,dx = \frac1{2\mu}\int_{V}\int_{V} |f(x)-f(y)|^2 \,dxdy .
\ee
Fix $x,y\in V$ and consider the set $U=B_r(x)\cap B_r(y)$.
Observe that $U$ contains the ball of radius $r-|xy|/2\ge r-\ep$
centered at the midpoint between $x$ and~$y$.
Hence $\vol(U) \ge C \nu_n (r-\ep)^n$.
For every $z\in U$ we have
$$
 |f(x)-f(y)|^2 \le 2 \bigl(|f(x)-f(z)|^2+|f(y)-f(z)|^2\bigr) .
$$
Therefore
$$
\begin{aligned}
 |f(x)-f(y)|^2 &\le \frac2{\vol(U)} \int_U \bigl(|f(x)-f(z)|^2+|f(y)-f(z)|^2\bigr)\,dz \\
 &\le  \frac2{\vol(U)} \bigl( F(x)+F(y) \bigr)
 \le \frac C{\nu_n(r-\ep)^n}\bigl( F(x)+F(y) \bigr)
\end{aligned}
$$
where
$$
 F(x) = \int_{B_r(x)}|f(x)-f(z)|^2\,dz .
$$
Plugging the last inequality into \eqref{e:dispersion1} yields
$$
\begin{aligned}
\int_{V} |f(x)-a|^2\,dx 
&\le \frac C{2\mu\nu_n(r-\ep)^n} \int_{V}\int_{V} \bigl( F(x)+F(y) \bigr) \,dxdy \\
&=\frac C{\nu_n(r-\ep)^n} \int_{V} F(x) \,dx
= \frac C{\nu_n(r-\ep)^n} E_r(f,V) .
\end{aligned}
$$
The lemma follows.
\end{proof}

\section{Discretization map and upper bound for $\la_k(\Gamma)$}
\label{sec:discretization}

Let $X=\{x_i\}_{i=1}^N\subset M$ and $\mu$
be as in Theorem~\ref{main}.
Recall that $\mu$ is a measure on $X$ and
$(X,\mu)$ $\ep$-approximates $(M,\vol)$
in the sense of Definition~\ref{d:vol-approx}.
We fix a partition $\{V_i\}_{i=1}^N$ of $M$ 
realizing this approximation,
that is, $V_i\subset B_\ep(x_i)$
and $\vol(V_i)=\mu_i:=\mu(x_i)$
for each $i$.
We assume that $\ep<\rho/n$.

\begin{definition} \label{P}
Define a {\it discretization map} $P\co L^2(M)\to L^2(X)$ by
$$
 Pf(x_i) = \mu_i^{-1}\int_{V_i} f(x)\,dx .
$$
In other words, $Pf(x_i)$ is the integral mean of $f|_{V_i}$.

We also need a map $P^*\co L^2(X)\to L^2(M)$ defined by
$$
 P^*u = \sum_{i=1}^N u(x_i) 1_{V_i}
$$
where $1_{V_i}$ is the characteristic function of the set $V_i$.
Here $P^*$ is the adjoint of $P$.
\end{definition}

From the Cauchy--Schwartz inequality one easily sees
that 
\be\label{e:Pf-L2}
\|Pf\| \le \|f\|_{L^2}
\ee for every $f\in L^2(M)$,
where the norm in the left-hand side is defined by \eqref{e:L2Xnorm}.
The definition implies that $P^*$ preserves the norm:
$$
 \|P^*u\|_{L^2} = \|u\|
$$
for all $u\in L^2(X)$, and is adjoint to $P$:
$$
 \langle f, P^*u\rangle_{L^2(M)} = \langle Pf, u\rangle_{L^2(X)}
$$
for all $u\in L^2(X)$, $f\in L^2(M)$.

\begin{lemma}\label{l:P*P}
Let $f\in H^1(M)$.
Then $\|f-P^*Pf\|_{L^2} \le Cn\ep \|df\|_{L^2}$.
\end{lemma}

\begin{proof}
We have
$$
 \|f-P^*Pf\|^2_{L^2} = \sum_i\int_{V_i} |f(x)-Pf(x_i)|^2\,dx .
$$
By Lemma~\ref{l:dispersion}, for every $r\in(\ep,2\rho)$ and every $i$ we have
$$
 \int_{V_i} |f(x)-Pf(x_i)|^2\,dx \le \frac C{\nu_n (r-\ep)^n} E_r (f,V_i) .
$$
Note that $\sum_i E_r(f,V_i)=E_r(f)$ by definition.
Therefore
$$
 \|f-P^*P f\|_{L^2}^2
 \le \frac C{\nu_n (r-\ep)^n} E_r(f)
 \le \frac C{n+2} \frac {r^n}{(r-\ep)^n} r^2 \|df\|^2_{L^2}
$$
where the last inequality follows from Lemma~\ref{l:Er-df-L2}.
Now let $r=(n+1)\ep$, then
$
\frac {r^n}{(r-\ep)^n} = \left(1+\frac1n\right)^n < 3 ,
$
hence
$$
\|f-P^*Pf\|_{L^2}^2 \le \frac C{n+2} r^2 \|df\|^2_{L^2}
\le Cn\ep^2 \|df\|^2_{L^2} .
$$
The lemma follows.
\end{proof}

\begin{lemma}
\label{l:P-properties}
Let $f\in H^1(M)$. Then
\begin{enumerate}
\item
$ \bigl| \|Pf\| - \|f\|_{L^2}\bigr| \le Cn\ep \|df\|_{L^2}$;
\item
$\|\delta(Pf)\| \le (1+\sigma)\|df\|_{L^2}$ 
where $\sigma=Cn(K\rho^2+\ep/\rho)$.
\end{enumerate}
\end{lemma}

\begin{proof}

(1)
Since $P^*$ preserves the norm, we have
$$
\bigl| \|Pf\| - \|f\|_{L^2}\bigr|
= \bigl| \|P^*Pf\|_{L^2}- \|f\|_{L^2}\bigr|
\le \|f-P^*Pf\|_{L^2} \le Cn\ep \|df\|_{L^2}
$$
where the last inequality follows from Lemma~\ref{l:P*P}.

(2) 
By \eqref{e:energy} we have
$$
\|\delta(Pf)\|^2 
= \frac{n+2}{\nu_n\rho^{n+2}} \sum_i\sum_{j:x_j\sim x_i} \mu_i\mu_j |Pf(x_j)-Pf(x_i)|^2 .
$$
The definition of $Pf$ implies that
$$
 Pf(x_j)-Pf(x_i) = \frac1{\mu_i\mu_j}\int_{V_i}\int_{V_j} (f(y)-f(x))\,dydx .
$$
Hence, by the Cauchy--Schwartz inequality,
$$
 |Pf(x_j)-Pf(x_i)|^2 \le \frac1{\mu_i\mu_j}\int_{V_i}\int_{V_j} |f(y)-f(x)|^2\,dydx .
$$
Therefore
$$
\begin{aligned}
\|\delta(Pf)\|^2 
&\le \frac{n+2}{\nu_n\rho^{n+2}} \sum_i\sum_{j:x_j\sim x_i} \int_{V_i}\int_{V_j} |f(y)-f(x)|^2\,dydx \\
&=\frac{n+2}{\nu_n\rho^{n+2}}\int_M\int_{U(x)} |f(y)-f(x)|^2\,dydx
\end{aligned}
$$
where the set $U(x)\subset M$ is defined as follows:
if $x\in V_i$, then $U(x)=\bigcup_{j:x_j\sim x_i} V_j$.
Note that $U(x)\subset B_{\rho+2\ep}(x)$.
Hence
$$
  \|\delta (Pf)\|^2
 \le \frac{n+2}{\nu_n\rho^{n+2}} E_{\rho+2\ep}(f)
$$
By Lemma~\ref{l:Er-df-L2},
$$
E_{\rho+2\ep}(f) \le \frac{\nu_n}{n+2} (\rho+2\ep)^{n+2}  (1+\sigma_1) \|df\|_{L^2}^2.
$$
where $\sigma_1=CnK\rho^2$.
Therefore
$$
\|\delta (Pf)\|^2 
\le (1+2\ep/\rho)^{n+2} (1+\sigma_1) \|df\|_{L^2}^2
\le (1+\sigma) \|df\|_{L^2}^2 .
$$
where $\sigma=Cn(K\rho^2+\ep/\rho)$.
This finishes the proof of Lemma~\ref{l:P-properties}.
\end{proof}

\begin{proposition}\label{p:upper-bound}
Let $\la_k=\la_k(M)$, $k\in\N$. Then
$$
\la_k(\Gamma) \le (1+\delta(\ep,\rho,\la_k))\la_k
$$
where
$$
 \delta(\ep,\rho,\la) = Cn(K\rho^2+\ep/\rho +\ep\sqrt\la),
$$
provided that $\ep\sqrt{\la_k}<c/n$.
Here $C$ and $c$ are absolute constants.
\end{proposition}

\begin{proof}
By the minimax principle, it suffices to show that there exists a
linear subspace $L\subset L^2(X)$ such that $\dim L=k$ and
$$
 \sup_{u\in L\setminus\{0\}} \frac{\|\delta u\|^2} {\|u\|^2}
 \le (1+\delta(\ep,\rho,\la_k))\la_k .
$$
Denote $\la=\la_k$.
Let $W\subset H^1(M)$ be the linear span of 
orthonormal eigenfunctions of $-\Delta_M$
corresponding to eigenvalues $\la_1,\dots,\la_k$.
For every $f\in W$, we have $\|df\|_{L^2}^2 \le \la \|f\|_{L^2}^2$.
By Lemma \ref{l:P-properties}(1) it follows that
$$
 \|Pf\| \ge \|f\|_{L^2} - Cn\ep\|df\|_{L^2} \ge  (1-Cn\ep\sqrt\la) \|f\|_{L^2}
$$
for every $f\in W$.
Hence $P|_W$ is injective if $\ep\sqrt\la<1/Cn$.
Let $L=P(W)$, then $\dim L=k$.
Pick $u\in L\setminus\{0\}$ and let $f\in W$ be such that $u=Pf$.
Then
$$
 \|u\|^2 \ge (1-Cn\ep\sqrt\la) \|f\|_{L^2}^2
$$
and, by Lemma~\ref{l:P-properties}(2),
$$
 \|\delta u\|^2 \le (1+\sigma_1)\|df\|_{L^2}^2
 \le (1+\sigma_1)\la\|f\|_{L^2}^2
$$
where $\sigma_1=Cn(K\rho^2+\ep/\rho)$.
Hence
$$
\frac{\|\delta u\|^2}{\|u\|^2}
\le
\frac{(1+\sigma_1)\la}{1-Cn\ep\sqrt\la}
\le
(1+\delta(\ep,\rho,\la))\la
$$
and the proposition follows.
\end{proof}

\section{Smoothening operator}
\label{sec:smoothening}

In this section we prepare technical tools
for the interpolation map, which is defined in the next section.
These tools are independent of the discretization.

Define a function $\psi\colon\R_+\to\R_+$ by
$$
 \psi(t) =
 \begin{cases}
  \frac{n+2}{2\nu_n} (1-t^2), &\qquad 0\le t\le 1, \\
  0, &\qquad t\ge 1 .
 \end{cases}
$$
The normalization constant $\frac{n+2}{2\nu_n}$
is chosen so that
\be\label{e:int-psi}
 \int_{\R^n} \psi(|x|)\,dx = 1.
\ee
Indeed, by \eqref{e:integralQ} we have
$$
 \int_{B_1(0)} (1-|x|^2)\,dx = \nu_n - \frac {n\nu_n}{n+2} = \frac{2\nu_n}{n+2} .
$$

Fix a positive $r<2\rho$ and
consider a kernel $k_r\colon M\times M\to\R_+$ defined by
$$
k_r(x,y)=r^{-n}\psi(r^{-1}|xy|)
$$
and the associated integral operator
$ \Lambda^0_r\colon L^2(M)\to C^{0,1}(M) $
given by
$$
 \Lambda^0_r f(x) = \int_M f(y) k_r(x,y)\, dy .
$$
Note that $k_r(x,y)=k_r(y,x)$ and
$$
|k_r(x,y)|\le \frac{Cn}{\nu_nr^n}
$$
for all $x,y\in M$.
A direct computation (using the derivative of the
distance function, see \eqref{e:derivative-of-distance})
yields
\be
\label{e:grad-kr}
 \grad k_r(\cdot,y)(x) = \frac{n+2}{\nu_nr^{n+2}}\exp_x^{-1}(y)
\ee
for $y\in B_r(x)$.

Define $\theta\in C^{0,1}(M)$ by $\theta=\Lambda^0_r(1_M)$.
If the metric of $M$ were flat,
we would have $\theta=1_M$ by \eqref{e:int-psi}.
The following lemma estimates $\|\theta-1_M\|$ in the Riemannian case.

\begin{lemma}
\label{l:theta-estimate}
For every $x\in M$, one has
$$
 (1+CnKr^2)^{-1} \le \theta(x) \le 1+CnKr^2
$$
and
$$
 |d_x\theta| \le Cn^2Kr .
$$
\end{lemma}

\begin{proof}
By definition,
$$
 \theta(x) = r^{-n} \int_{B_r(x)} \psi(r^{-1}d(x,y))\,dy
  = r^{-n} \int_{B_r(0)\subset T_xM} \psi(r^{-1} |v|) J_x(v)\,dv
$$
where $J_x(v)$ is the Jacobian of the Riemannian exponential map,
see Section \ref{sec:riem-prelim}.
Since the integral of $\psi$ equals 1, the Jacobian estimate \eqref{e:rauch-jacobian}
implies the first assertion of the lemma.

To estimate $|d_x\theta(x)|$, we compute it using \eqref{e:grad-kr}:
$$
 \grad \theta(x) = \frac{n+2}{\nu_nr^{n+2}} \int_{B_r(x)} \exp_x^{-1}(y)\,dy
 = \frac{n+2}{\nu_nr^{n+2}}\int_{B_r(0)\subset T_xM} v J_x(v)\,dv .
$$
Since $\int_{B_r(0)} v\,dv=0$, one can replace $J_x(v)$ in the last integral
by $J_x(v)-1$. Then the Jacobian estimate \eqref{e:rauch-jacobian} implies that
$$
 |d_x\theta(x)|=|\grad \theta(x)|
 \le \frac{n+2}{\nu_nr^{n+2}}\int_{B_r(0)\subset T_xM} |v|\cdot CnKr^2 \,dv
 \le Cn^2Kr 
$$
(the last inequality  follows from the relations $|v|\le r$ and $\vol(B_r(0))=\nu_nr^n$).
\end{proof}

\begin{definition}
\label{d:smoothening}
Now we define a bounded operator $\Lambda_r\co L^2(M)\to C^{0,1}(M)$ by
$$
 \Lambda_rf = \theta^{-1} \cdot \Lambda^0_r f .
$$
\end{definition}

The factor $\theta^{-1}$ ensures that $\Lambda_r$ preserves the 
subspace of constants.

\begin{lemma}
\label{l:norm-Lambdaf}
For every $f\in L^2(M)$ one has
$$
 \|\Lambda_rf\|_{L^2} \le (1+\sigma) \|f\|_{L^2}
$$
where $\sigma=CnKr^2$.
\end{lemma}

\begin{proof}
This is a standard estimate.
For every $x\in M$ we have
$$
\begin{aligned}
 |\Lambda_r^0 f(x)|^2 
 &= \left| \int_M  f(y)k_r(x,y) \,dy\right|^2
 \le \left(\int_M k_r(x,y)\,dy\right)\left(\int_M |f(y)|^2 k_r(x,y)\,dy\right) \\
 &= \theta(x) \int_M |f(y)|^2 k_r(x,y)\,dy 
\end{aligned}
$$
by the Cauchy--Schwartz inequality. Hence
$$
\begin{aligned}
 |\Lambda_r f(x)|^2 &= \theta(x)^{-2} |\Lambda_r^0 f(x)|^2 
 \le \theta(x)^{-1} \int_M |f(y)|^2 k_r(x,y)\,dy \\
 &\le (1+\sigma) \int_M |f(y)|^2 k_r(x,y)\,dy
\end{aligned}
$$
by Lemma \ref{l:theta-estimate}.
Integrating this inequality over $M$ yields
$$
 \|\Lambda_r f(x)\|_{L^2}^2
 \le (1+\sigma) \int_M |f(y)|^2 \int_M k_r(x,y)\,dxdy
 \le (1+\sigma)^2 \|f\|_{L^2}^2
$$
since for every $y\in M$ we have
$$
\int_M k_r(x,y)\,dx = \theta(y) \le 1+\sigma
$$ 
by Lemma \ref{l:theta-estimate}. The lemma follows.
\end{proof}

\begin{lemma}
\label{l:Lambdaf-f}
For every $f\in L^2(M)$ one has
\begin{equation} \label{22.2}
 \|\Lambda_rf-f\|^2_{L^2} \le \frac{Cn}{\nu_nr^n} E_r(f) .
\end{equation}
\end{lemma}

\begin{proof}
We fix a particular function $f\co M\to\R$ representing the given element
of $L^2(M)$, so we can speak about pointwise values of~$f$.
Fix $x\in M$ and let $a=f(x)$.
Since $\Lambda_r$ preserves the constants, we have
$$
\begin{aligned}
 \Lambda_rf(x)- f(x) &= \Lambda_rf(x)- a = \Lambda_r(f-a\cdot 1_M)(x) \\
 &= \theta^{-1}(x) \int_{B_r(x)} (f(y)-a) k_r(x,y)\,dy \\
 &= \theta^{-1}(x) \int_{B_r(x)} (f(y)-f(x)) k_r(x,y)\,dy .
\end{aligned}
$$
By the Cauchy--Schwartz inequality, it follows that
$$
\begin{aligned}
 |\Lambda_rf(x)- f(x)|^2
 &\le \theta^{-2}(x) \left( \int_{B_r(x)} k_r(x,y)\,dy \right)
 \left(\int_{B_r(x)} |f(y)-f(x)|^2 k_r(x,y)\,dy\right) \\
 &=\theta^{-1}(x)\int_{B_r(x)} |f(y)-f(x)|^2 k_r(x,y)\,dy \\
 &\le \frac{Cn}{\nu_nr^n} \int_{B_r(x)} |f(y)-f(x)|^2 
\end{aligned}
$$
since $|\theta^{-1}(x)|\le C$ (cf.\ Lemma \ref{l:theta-estimate})
and $|k_r(x,y)|\le \frac{Cn}{\nu_nr^n}$.
Integrating this inequality yields the result.
\end{proof}

\begin{lemma}
\label{l:dLambdaf}
For every $f\in L^2(X)$ one has
$$
 \|d(\Lambda_rf)\|_{L^2}^2 \le \big(1+Cn^2Kr^2\big)\frac{n+2}{\nu_nr^{n+2}} E_r(f) .
$$
\end{lemma}

\begin{proof}
We fix a particular function $f\co M\to\R$ representing the given element
of $L^2(M)$, so we can speak about pointwise values of~$f$.
Denote $\Lambda_rf$ by $\tilde f$.
For any constant $a\in\R$ we have
$$
 \tilde f(x) = a + \theta^{-1}(x) \int_{B_r(x)} (f(y)-a) k_r(x,y)\,dy
$$
for every $x\in M$.
Differentiating this identity yields
$$
 d_x\tilde f =  \theta^{-1}(x) \int_{B_r(x)} (f(y)-a) d_xk_r(\cdot,y) \,dy 
 + d_x(\theta^{-1})\int_{B_r(x)} (f(y)-a) k_r(x,y) \,dy .
$$
Substituting $a=f(x)$ yields
$$
 d_x\tilde f = \theta^{-1}(x)A_1(x)+A_2(x)
$$
where 
$$
 A_1(x) = \int_{B_r(x)} (f(y)-f(x)) d_xk_r(\cdot,y) \,dy 
$$
and
$$
 A_2(x) =d_x(\theta^{-1})\int_{B_r(x)} (f(y)-f(x)) k_r(x,y) \,dy.
$$
Since $|\theta^{-1}(x)|\le 1+CnKr^2$ for all $x\in M$
(cf.\ Lemma~\ref{l:theta-estimate}),
we have
\be\label{e:detildef}
 \|d\tilde f\|_{L^2} = \|\theta^{-1}A_1+A_2\|_{L^2} \le (1+CnKr^2)\|A_1\|_{L^2}+\|A_2\|_{L^2} .
\ee

Let us first estimate $\|A_2\|_{L^2}$.
By the Cauchy--Schwartz inequality,
$$
\begin{aligned}
 |A_2(x)|^2 
 &\le |d_x(\theta^{-1})|^2
 \left(\int_{B_r(x)} k_r(x,y) \,dy \right)
 \left(\int_{B_r(x)} |f(y)-f(x)|^2 k_r(x,y) \,dy\right) \\
 &= |d_x(\theta^{-1})|^2\, \theta(x)\int_{B_r(x)} |f(y)-f(x)|^2 k_r(x,y) \,dy \\
 &\le \frac{Cn^5K^2}{\nu_nr^{n-2}} \int_{B_r(x)} |f(y)-f(x)|^2 \,dy 
\end{aligned}
$$
since $\theta(x)\le C$, $|d_x(\theta^{-1})|\le Cn^2Kr$ (cf.\ Lemma~\ref{l:theta-estimate})
and $|k_r(x,y)|\le \frac{Cn}{\nu_nr^n}$.
Integrating this inequality yields
$$
 \|A_2\|^2_{L^2} \le \frac{Cn^5K^2}{\nu_nr^{n-2}} E_r(f) .
$$
We rewrite this inequality as follows:
\be\label{e:A2-estimate}
 \|A_2\|_{L^2}
 \le Cn^2Kr^2 \sqrt{\frac{n+2}{\nu_nr^{n+2}} E_r(f)} .
\ee

Now let us estimate $A_1$.
Fix $x\in M$. Recall that
$$
 |A_1(x)| = \max \{ \langle A_1(x),w\rangle : w\in T_xM, |w|=1 \}
$$
where the angle brackets denote the standard pairing of co-vectors and vectors.
Let $w\in T_xM$ be a unit vector realizing this maximum.
Then $|A_1(x)|=\langle A_1(x),w\rangle$.

Plugging the expression \eqref{e:grad-kr}
for the derivative of~$k_r$ into the definition of $A_1$ yields 
$$
\begin{aligned}
 |A_1(x)|=\langle A_1(x),w\rangle 
 &= \frac{n+2}{\nu_nr^{n+2}}\int_{B_r(x)} (f(y)-f(x)) \langle\exp_x^{-1}(y),w\rangle \,dy \\
 &= \frac{n+2}{\nu_nr^{n+2}}\int_{B_r(0)\subset T_xM} \phi(v) \langle v,w\rangle J_x(v) \,dv
\end{aligned}
$$
where $\phi(v)=f(\exp_x(v))-f(x)$.
For brevity, we denote the ball $B_r(0)\subset T_xM$ by~$B$.
By the Cauchy--Schwartz inequality, it follows that
$$
 |A_1(x)|^2 \le \left(\frac{n+2}{\nu_nr^{n+2}}\right)^2
  \left(\int_{B} |\phi(v)|^2 J_x(v)^2\,dv\right)
  \left(\int_{B} \langle v,w\rangle^2\,dv\right) .
$$
Since $|w|=1$, we have
$$
 \frac{n+2}{\nu_nr^{n+2}} \int_{B} \langle v,w\rangle^2\,dv = 1,
$$
hence the above inequality boils down to
$$
\begin{aligned}
 |A_1(x)|^2 
 &\le \frac{n+2}{\nu_nr^{n+2}}
  \int_{B} |\phi(v)|^2 J_x(v)^2\,dv \\
 &= \frac{n+2}{\nu_nr^{n+2}}
  \int_{B_r(x)} |f(y)-f(x)|^2 J_x(\exp_x^{-1}(y))\,dy \\
 &\le \big(1+CnKr^2\big) \frac{n+2}{\nu_nr^{n+2}} \int_{B_r(x)} |f(y)-f(x)|^2 \,dy 
\end{aligned}
$$
where the last inequality follows from the Jacobian estimate \eqref{e:rauch-jacobian}.

Integrating this inequality with respect to $x$ over $M$ yields
$$
 \|A_1\|_{L^2}^2 \le \big(1+CnKr^2\big) \frac{n+2}{\nu_nr^{n+2}} E_r(f) .
$$
This, \eqref{e:detildef} and \eqref{e:A2-estimate} imply that
$$
\begin{aligned}
 \|d\tilde f\|_{L^2} 
 &\le \left( (1+CnKr^2)^{3/2} + Cn^2Kr^2 \right) \sqrt{\frac{n+2}{\nu_nr^{n+2}} E_r(f)} \\
 &\le \left( 1+ Cn^2Kr^2 \right) \sqrt{\frac{n+2}{\nu_nr^{n+2}} E_r(f)} .
\end{aligned}
$$
The lemma follows.
\end{proof}

\section{Interpolation map and lower bound for $\la_k(\Gamma)$}
\label{sec:interpol}

\begin{definition} \label{I}
Define the \textit{interpolation map}
$I\co L^2(X)\to C^{0,1}(M)$ by
$$
 Iu = \Lambda_{\rho-2\ep} P^*u
$$
where $\Lambda_{\rho-2\ep}$ is the smoothening operator
defined in the previous section 
(see Definition~\ref{d:smoothening})
and $P^*\co L^2(X)\to L^2(M)$
is defined in Definition \ref{P}.
\end{definition}

Lemma \ref{l:norm-Lambdaf} and the fact that $P^*$ preserves the norm
imply that
\be\label{e:If-L2}
 \|Iu\|_{L^2} \le (1+CnK\rho^2) \|u\| \le C\|u\|
\ee
for every $u\in L^2(X)$.

\begin{lemma}
\label{l:interpolation}
For every $u\in L^2(X)$ one has
\begin{enumerate}
\item
 $\bigl|\|Iu\|_{L^2} - \|u\|\bigr| \le C\rho\|\delta u\|$;
\item
 $\|d (Iu)\|_{L^2} \le (1+\sigma) \|\delta u\|$
 where $\sigma=Cn^2K\rho^2+Cn\ep/\rho$.
\end{enumerate}
\end{lemma}

\begin{proof}
(1)
Since $\|P^*u\|_{L^2}=\|u\|$, we have
$$
 \bigl|\|Iu\|_{L^2} - \|u\|\bigr| = \bigl|\|Iu\|_{L^2} - \|P^*u\|_{L^2}\bigr|
 \le \|Iu-P^*u\|_{L^2} .
$$
By Lemma \ref{l:Lambdaf-f},
$$
\|Iu-P^*u\|_{L^2}^2 = \|\Lambda_{\rho-2\ep}P^*u-P^*u\|_{L^2}^2
\le \frac{Cn}{\nu_n(\rho-2\ep)^n} E_{\rho-2\ep}(P^*u) .
$$
Since $\ep<\rho/n$ and $\ep<\rho/3$,
we have $(\rho-2\ep)^{-n} \le C \rho^{-n}$, hence
\be\label{e:If-barf}
\|Iu-P^*u\|_{L^2}^2
\le \frac{Cn}{\nu_n\rho^n} E_{\rho-2\ep}(P^*u) .
\ee
Let us estimate $E_{\rho-2\ep}(P^*u)$ in terms of $\delta u$.
By definition, 
$$
\begin{aligned}
\|\delta u\|^2 
&= \frac{n+2}{\nu_n\rho^{n+2}} \sum_i \sum_{j:x_j\sim x_i} \mu_i\mu_j |u(x_j)- u(x_i)|^2 \\
&= \frac{n+2}{\nu_n\rho^{n+2}} \int_M\int_{U(x)} |P^*u(y)-P^*u(x)|^2\,dy\,dx
\end{aligned}
$$
where sets $U(x)\subset M$ are defined as follows:
if $x\in V_i$, then $U(x)=\bigcup_{j:x_j\sim x_i} V_j$.
Since $U(x)\supset B_{\rho-2\ep}(x)$, we have
$$
 \|\delta u\|^2
 \ge \frac{n+2}{\nu_n\rho^{n+2}} \int_M\int_{B_{\rho-2\ep}(x)} |P^*u(y)-P^*u(x)|^2\,dy\,dx
 = \frac{n+2}{\nu_n\rho^{n+2}} E_{\rho-2\ep}(P^*u) .
$$
Thus
\be\label{e:Erho-vs-delta}
E_{\rho-2\ep}(P^*u) \le \frac{\nu_n\rho^{n+2}}{n+2} \|\delta u\|^2 .
\ee
This and \eqref{e:If-barf} imply that
$$
 \|Iu-P^*u\|_{L^2}^2
 \le \frac{Cn}{\nu_n\rho^n}\frac{\nu_n\rho^{n+2}}{n+2} \|\delta u\|^2 
 \le C\rho^2 \|\delta u\|^2 .
$$
Hence
\be
\label{e:If-barf-L2}
 \|Iu-P^*u\|_{L^2} \le C\rho \|\delta u\|
\ee
and the first assertion of the lemma follows.

\medskip
(2) By Lemma \ref{l:dLambdaf},
$$
 \|d(Iu)\|_{L^2}^2 = \|d(\Lambda_{\rho-2\ep}P^*u)\|_{L^2}^2
 \le (1+\sigma_1)\frac{n+2}{\nu_n(\rho-2\ep)^{n+2}} E_{\rho-2\ep}(P^*u) 
$$
where $\sigma_1=Cn^2K\rho^2$.
By \eqref{e:Erho-vs-delta},
$$
 \frac{n+2}{\nu_n(\rho-2\ep)^{n+2}} E_{\rho-2\ep}(P^*u)
 \le \left( \tfrac{\rho}{\rho-2\ep} \right)^n \|\delta u\|^2
 \le (1+\sigma_2) \|\delta u\|^2 .
$$
where $\sigma_2=Cn\ep/\rho$.
Thus
$$
\|d(Iu)\|_{L^2}^2 
\le (1+\sigma_1)(1+\sigma_2) \|\delta u\|^2 
\le (1+\sigma) \|\delta u\|^2
$$
where $\sigma=Cn^2K\rho^2+Cn\ep/\rho$.
The second assertion of the lemma follows.
\end{proof}

\begin{proposition}\label{p:lower-bound}
Let $\la_k=\la_k(M)$, $k\in\N$. Then
$$
\la_k(\Gamma) \ge (1-\delta(\ep,\rho,\la_k))\la_k
$$
where
$$
 \delta(\ep,\rho,\la) = C (n^2K\rho^2+n\ep/\rho+\rho\sqrt\la)
$$
provided that $\rho\sqrt{\la_k}<c_0$.
Here $C$ and $c_0>0$ are absolute constants.
\end{proposition}

\begin{proof}
With Lemma \ref{l:interpolation} at hand, the proof is similar
to that of Proposition \ref{p:upper-bound}.
Let $\la=\la_k(\Gamma)$.
We assume that $\la<\la_k(M)$,
otherwise there is nothing to prove.
By the minimax principle,
it suffices to show that there exists a
linear subspace $L\subset H^1(M)$ such that ${\dim L=k}$
and
$$
 \sup_{f\in L\setminus\{0\}} \frac{\|df\|_{L^2}^2}{\|f\|_{L^2}^2} 
 \le  (1-\delta(\ep,\rho,\la))^{-1}\la.
$$
Indeed, this inequality would imply that $\lambda_k(M)\le (1-\delta(\ep,\rho,\la))^{-1}\la$
and therefore
$$
 \lambda \ge (1-\delta(\ep,\rho,\la))\lambda_k\ge (1-\delta(\ep,\rho,\la_k))\lambda_k .
$$
Let $W\subset L^2(X)$ be the linear span of
$k$ orthonormal eigenvectors of $-\Delta_\Gamma$
corresponding to eigenvalues $\la_1(\Gamma),\dots,\la_k(\Gamma)$.
For every $u\in W$, we have $\|\delta u\|^2 \le \la \|u\|^2$.
By Lemma \ref{l:interpolation}(1) it follows that
$$
 \|Iu\|_{L^2} 
 \ge \|u\| - C\rho\|\delta u\|
 \ge  (1-C\rho\sqrt\la) \|u\|
$$
for every $u\in W$.
Hence $I|_W$ in injective if $C\rho\sqrt\la<1$.
Let $L=I(W)$, then $\dim L=k$.
Pick $f\in L\setminus\{0\}$ and let $u\in W$ be such that $f=Iu$.
Then
$$
 \|f\|_{L^2}^2 \ge (1-C\rho\sqrt\la) \|u\|^2
$$
and, by Lemma~\ref{l:interpolation}(2),
$$
 \|df\|_{L^2}^2 \le (1+\sigma_1)\|\delta u\|^2
 \le (1+\sigma_1)\la \|u\|^2 .
$$
where $\sigma_1=1+Cn^2K\rho^2+Cn\ep/\rho$.
Hence
$$
\frac{\|df\|_{L^2}^2}{\|f\|_{L^2}^2}
\le
\frac{(1+\sigma_1)\la}{1-C\rho\sqrt\la}
\le
(1-\delta(\ep,\rho,\la))^{-1}\la
$$
and the proposition follows.
\end{proof}

We conclude this section by showing that the operators $P$ and $I$
are almost inverse to each other
(at bounded energy levels).

\begin{lemma} \label{l:inverse}
{\it 1.} Let $f \in H^1(M)$. Then
$$
\|I P f-f\|_{L^2}\leq C \rho\,\|df\|_{L^2} .
$$
{\it 2.} Let $u \in L^2(X)$. Then
$$
\| P I u-u\| \leq C \rho\,\|\delta u\| .
$$
\end{lemma}

\begin{proof} {\it 1.}
Let $\bar f=P^*Pf$. By definition,
$$
\|I P f-f\|_{L^2}= \|\Lambda_{\rho-2\ep} \bar f- f\|_{L^2} 
 \leq \|\Lambda_{\rho-2\ep} (\bar f-f)\|_{L^2}+
\|\Lambda_{\rho-2\ep}f -f\|_{L^2}.
$$
Lemma \ref{l:norm-Lambdaf} and Lemma \ref{l:P*P} imply that
$$
 \|\Lambda_{\rho-2\ep}(\bar f - f)\|_{L^2} \leq C \|\bar f- f\|_{L^2} \leq Cn \ep \|df\|_{L^2} .
$$
Next, by Lemma \ref{l:Lambdaf-f} and Lemma \ref{l:Er-df-L2},
$$
\|\Lambda_{\rho-2\ep}f -f\|^2_{L^2} \le \frac{Cn}{\nu_n(\rho-2\ep)^n} E_{\rho-2\ep}(f) \leq C \rho^2 \|df\|^2.
$$
Combining the above inequalities and using the fact that $\ep <\rho/n$,
we obtain the first assertion of the lemma.

{\it 2.} 
Since $P^*$ preserves the norm, we have
$$
\| PIu-u\| =\|P^*(PIu-u)\|
\leq \|P^*PIu- Iu\|_{L^2} + \| Iu - P^*u\|_{L^2} .
$$
By Lemma \ref{l:P*P},
$$
\|P^*PIu- Iu\|_{L^2} \leq Cn \ep \|d (Iu)\|_{L^2}
\leq Cn \ep \|\delta u\| ,
$$
where at the last stage we use Lemma \ref{l:interpolation}(1).
By \eqref{e:If-barf-L2},
$$
 \|Iu-P^*u\|_{L^2} \le C\rho \|\delta u\|
$$
As $\ep <\rho/n$, the above inequalities imply 
the second assertion of the lemma.
\end{proof}

\section{Eigenfunction approximation and proof of theorems}
\label{sec:eigenfunction}

To prove  Theorem \ref{main}, first observe that
the estimate \eqref{main_th} follows from Propositions \ref{p:upper-bound} and \ref{p:lower-bound}.
Next recall that, as follows from \cite{Anderson},
the space $\M$ is pre-compact in Lipschitz topology.
Therefore the eigenvalue $\la_k(M)$ is uniformly bounded for all $M\in\M$,
that is, $\la_k(M)\le C_{\M,k}$.
Using this fact, we obtain the second estimate in Theorem~\ref{main}
from the first one.

To proceed with the eigenfunctions approximations,
we introduce some notation.
For an interval $J\subset\R$, denote by 
$H_J(M)$ the subspace of $H^1(M)$ spanned by
the eigenfunctions with eigenvalues from~$J$.
In particular, $H_{\{\la\}}(M)$ is the eigenspace
associated to an eigenvalue~$\la$.
We abbreviate $H_{(-\infty,\la)}(M)$ by $H_\la(M)$.
We use similar notation $H_J(X)$ and $H_\la(X)$
for subspaces of $L^2(X)$ spanned by eigenvectors
of $-\Delta_\Gamma$.

Note that the dimension of $H_\la(M)$ is uniformly bounded
over $M\in\M$ (for every fixed $\la$), see \cite[Theorem~3]{BeBeGa}.

We denote by $\mathbb P_J$ the orthogonal projector from $L^2(M)$ to $H_J(M)$.
Note that $\mathbb P_J$ does not increase the Dirichlet energy norm.
Similarly to the above notation,
we abbreviate $\mathbb P_{(-\infty,\la)}$ by $\mathbb P_\la$.
We use the same notation $\mathbb P_J$ and $\mathbb P_\la$
for orthogonal projectors from $L^2(X)$ to $H_J(X)$ and $H_\la(X)$.

\begin{lemma}\label{l:inverse-strong}
1. Let $\la>0$ and $f\in H_\la(M)$. Then
$$
\|\delta(Pf)\| \ge (1-\sigma)\|df\|_{L^2} 
$$
where $\sigma=C(\rho\sqrt\la+n^2K\rho^2+n\ep/\rho)$.

2. Let $\la>0$ and $u\in H_\la(X)$. Then
$$
\|d(Iu)\|_{L^2} \ge (1-\sigma)\|\delta u\|
$$
where $\sigma=C(\rho\sqrt\la+nK\rho^2+n\ep/\rho)$
\end{lemma}

\begin{proof}
{\it 1.}
First we are going to estimate $\|d(IPf)\|_{L^2}$ from below
in terms of $\|df\|_{L^2}$.
Since the projector $\mathbb P_\lambda\co L^2(M)\to H_\lambda(M)$
does not increase the Dirichlet energy,
$$
 \|d(IPf)\|_{L^2} \ge \|d(\mathbb P_\la IPf)\|_{L^2}
 \ge \|df\|_{L^2} - \|d(\mathbb P_\la IPf-f)\|_{L^2} .
$$
Since $f\in H_\la(M)$, we have
$$
\begin{aligned}
\|d(\mathbb P_\la IPf-f)\|_{L^2} &= \|d(\mathbb P_\la (IPf-f))\|_{L^2}
\le \sqrt\la \|\mathbb P_\la (IPf-f)\|_{L^2} \\
&\le \sqrt\la \|IPf-f\|_{L^2}  \le C\rho\sqrt\la \,\|df\|_{L^2}.
\end{aligned}
$$
where the first inequality follows from the fact that
$\|dg\|_{L^2} \le \sqrt\la \|g\|_{L^2}$ for every $g\in H_\la(M)$
and the last one from Lemma \ref{l:inverse}.
Thus
$$
 \|d(IPf)\|_{L^2} \ge (1-\sigma_1)\|df\|_{L^2}
$$
where $\sigma_1=C\rho\sqrt\la$.
By Lemma \ref{l:interpolation}(2),
$$
 \|d(IPf)\|_{L^2} \le (1+\sigma_2) \|\delta(Pf)\|
$$
where $\sigma_2=C(n^2K\rho^2+n\ep/\rho)$.
Thus
$$
 \|\delta(Pf)\| \ge (1+\sigma_2)^{-1}(1-\sigma_1)\|df\|_{L^2}
$$
and the first assertion of the lemma follows.

{\it 2.}
The proof of the second assertion is completely similar.
Just interchange the roles of $M$ and $X$ and use
Lemma \ref{l:P-properties}(2) rather than Lemma \ref{l:interpolation}(2)
at the final step.
\end{proof}

We need the following simple lemma from linear algebra.

\begin{lemma}
\label{l:supQ-Q'}
Let $L$ be a finite-dimensional Euclidean space and $k=\dim L$.
Let $Q$ and $Q'$ be quadratic forms on $L$
and $\la_1\le\dots\le\la_k$ and $\la'_1\le\dots\la'_k$
their respective eigenvalues.
Suppose that $Q\ge Q'$. Then
$$
 \sup_{v\in L, \|v\|=1} \{ Q(v)-Q'(v) \} \le k\max_{1\le j\le k} \{\la_j-\la'_j\} .
$$
\end{lemma}

\begin{proof}
The left-hand side is the largest eigenvalue of the quadratic form $Q-Q'$.
Since $Q-Q'$ is nonnegative, its largest eigenvalue is bounded above
by it trace. On the other hand,
$$
 \trace(Q-Q')  = \trace(Q)-\trace(Q') = \sum_{j=1}^k (\la_j-\la'_j)
 \le k\max_{1\le j\le k} \{\la_j-\la'_j\} ,
$$
hence the result.
\end{proof}

We fix orthonormal eigenfunctions $\{f_k\}_{k=1}^\infty$ of $-\Delta_M$
and orthonormal eigenvectors $\{u_k\}_{k=1}^N$ of $-\Delta_X$.

\begin{lemma}
\label{l:highfreq}
1. Let $\la=\la_k(M)$. Then for every $a>0$,
$$
 \|Pf_k - \mathbb P_{\la+a}Pf_k\|^2 \le C_{\M,k}a^{-1}(\rho+\ep/\rho)
$$
and
$$
 \|\delta(Pf_k - \mathbb P_{\la+a}Pf_k)\|^2 \le C_{\M,k}(1+a^{-1})(\rho+\ep/\rho)
$$
provided that $\rho+\ep/\rho< C_{\M,k}^{-1}$.

2. Let $\la=\la_k(X)$. Then for every $a>0$,
$$
 \|Iu_k - \mathbb P_{\la+a}Iu_k\|^2_{L^2} \le C_{\M,k}a^{-1}(\rho+\ep/\rho)
$$
and
$$
 \|d(Iu_k - \mathbb P_{\la+a}Iu_k)\|^2_{L^2} \le C_{\M,k}(1+a^{-1})(\rho+\ep/\rho)
$$
provided that $\rho+\ep/\rho< C_{\M,k}^{-1}$.
\end{lemma}

\begin{proof}
{\it 1.} Let $W$ be the linear span of $f_1,\dots,f_k$
and $L=P(W)\subset L^2(X)$.
As in the proof of Proposition~\ref{p:upper-bound},
we have $\dim L=k$ if $\rho+\ep/\rho$ is sufficiently small.
Let $Q$ denote the discrete Dirichlet energy form on $L^2(X)$,
and let $\la_1^L\le\dots\le\la_k^L$ be the eigenvalues of $Q|_L$
(with respect to the Euclidean structure on $L$ defined by the
restriction of the $L^2(X)$ norm).

Recall that $\lambda_k(M)\le C_{\M,k}$.
This and Lemma~\ref{l:P-properties} imply that
for every $f\in W$,
$$
  (1-\sigma)\|f\|_{L^2} \le \|Pf\| \le (1+\sigma)\|f\|_{L^2}
$$
and
$$
    \|\delta(Pf)\|  \le (1+\sigma)\|df\|_{L^2}
$$
where $\sigma=C_{\M,k}(\rho+\ep/\rho)$.
By the minimax principle it follows that
\be\label{e:la^L}
 \la_j^L \le \bigl(\tfrac{1+\sigma}{1-\sigma}\bigr)^2 \la_j(M)
 \le \la_j(M) + C_{\M,k}(\rho+\ep/\rho)
\ee
for all $j\le k$, provided that $\sigma<1/2$.

Now define another quadratic form $Q'$ on $L^2(X)$ by
$$
 Q'(u) = Q(\mathbb P_{\lambda+a}(u)) + \la\|u-\mathbb P_{\lambda+a}(u)\|^2 .
$$
Clearly $Q'\le Q$.
The eigenvectors $u_1,u_2,\dots$ of $Q$ are also eigenvectors of $Q'$ and the corresponding
eigenvalues are $\la_1(\Gamma),\la_2(\Gamma),\dots,\la_m(\Gamma),\la,\la,\dots$,
where $m$ is the largest integer such that $\la_m(\Gamma)<\la+a$.
Therefore for every $j\le m$ and every $j$-dimensional subspace
$V\subset L^2(X)$ we have
\be\label{e:Q'}
 \sup_{v\in V\setminus\{0\}} \frac{Q'(v)}{\|v\|^2} \ge \min\{\la,\la_j(\Gamma)\} .
\ee
Indeed, $V$ has a nontrivial intersection with
the orthogonal complement of the linear span of $u_1,\dots,u_{j-1}$,
and any vector $v$ from this intersection
satisfies $Q'(v) \ge \min\{\la,\la_j(\Gamma)\}\|v\|^2$.
Let $\la'_1\le\dots\le\la'_k$ be the eigenvalues of $Q'|_L$
(with respect to the restriction of the $L^2(X)$ norm to~$L$).
Then \eqref{e:Q'} and the minimax principle imply that
$\la'_j\ge\min\{\la,\la_j(\Gamma)\}$ for all $j\le k$.
By Theorem~\ref{main} it follows that
$$
 \la'_j \ge \la_j(M) - C_{\M,k}(\rho+\ep/\rho) .
$$
and hence, by \eqref{e:la^L},
$$
 \la_j^L-\la'_j \le C_{\M,k}(\rho+\ep/\rho) 
$$
for all $j\le k$. This and Lemma \ref{l:supQ-Q'} imply that
\be\label{e:Qu-Q'u}
 Q(u)-Q'(u) \le C_{\M,k}(\rho+\ep/\rho) \|u\|^2
\ee
for every $u\in L$.

Let $u\in L$ and $u'=u-\mathbb P_{\la+a}u$.
Since $Q'(\mathbb P_{\la+a}u)=Q(\mathbb P_{\la+a}u)$, we have
\be\label{e:Qu-Q'u2}
 Q(u)-Q'(u) = Q(u')-Q'(u') = Q(u')-\la\|u'\|^2
 \ge\frac a{\la+a} Q(u')
\ee
where the last inequality follows from the fact that
$Q(u')\ge(\la+a)\|u'\|^2$
since  $u'\in H_{[\la+a,+\infty)}(X)$.
Now \eqref{e:Qu-Q'u} and \eqref{e:Qu-Q'u2} imply that
$$
 Q(u') \le \frac{\la+a}{a} C_{\M,k}(\rho+\ep/\rho)\|u\|^2
$$
and therefore
$$
 \|u'\|^2 \le (\la+a)^{-1} Q(u') 
 \le a^{-1}C_{\M,k}(\rho+\ep/\rho)\|u\|^2 .
$$
Substituting $u=Pf_k$ into the last two inequalities
and taking into account that $\|Pf_k\| \le 1+\sigma <2$
yields the first assertion of the lemma.

{\it 2.}
The proof of the second assertion is similar.
Just interchange the roles of $M$ and $X$
and use Lemma~\ref{l:interpolation}
rather than Lemma \ref{l:P-properties}.
\end{proof}

\begin{lemma}
\label{l:eigenspace}
1. Let $\la=\la_k(M)$ and let $\alpha,\beta,\gamma>0$ be such that
$\alpha\le\beta\le\gamma\le 1$ and
the interval $(\la+\alpha,\la+\beta)$ does not contain eigenvalues of $-\Delta_\Gamma$.
Then 
$$
 \| Pf_k - \mathbb P_{(\la-\gamma,\la+\alpha]}Pf_k\|^2
 \le C\alpha\gamma^{-1} + C_{\M,k}\beta^{-1}\gamma^{-1}(\rho+\ep/\rho)
$$
provided that $\rho+\ep/\rho< C_{\M,k}^{-1}$.

2. Let $\la=\la_k(\Gamma)$ and let $\alpha,\beta,\gamma>0$ be such that
$\alpha\le\beta\le\gamma\le 1$ and
the interval $(\la+\alpha,\la+\beta)$ does not contain eigenvalues of $-\Delta_M$.
Then 
$$
 \| Iu_k - \mathbb P_{(\la-\gamma,\la+\alpha]}Iu_k\|^2_{L^2}
 \le C\alpha\gamma^{-1} + C_{\M,k}\beta^{-1}\gamma^{-1}(\rho+\ep/\rho)
$$
provided that $\rho+\ep/\rho< C_{\M,k}^{-1}$.
\end{lemma}

\begin{proof}
{\it 1.}
As in the previous lemma, we denote the discrete Dirichlet energy form by~$Q$.
Let $u=Pf_k$.
Decompose $u$ into the sum of three orthogonal vectors
$u=u_0+u_-+u_+$ where $u_0\in H_{(\la-\gamma,\la+\alpha]}(X)$,
$u_-\in H_{(-\infty,\la-\gamma]}(X)$
and $u_+\in H_{(\la+\alpha,+\infty)}(X)$.
Note that $u_+\in H_{[\la+\beta,+\infty)}(X)$
due to our assumption about eigenvalues of $-\Delta_\Gamma$.
Applying Lemma~\ref{l:highfreq} with $\beta$ in place of $a$
yields that
\be\label{e:u+}
 \|u_+\|^2 \le C_{\M,k}\beta^{-1}(\rho+\ep/\rho)
\ee
and
$$
 Q(u_+) \le C_{\M,k}\beta^{-1}(\rho+\ep/\rho) .
$$
By Lemma~\ref{l:inverse-strong},
$$
 Q(u) = \|\delta(Pf_k)\|^2 \ge (1-\sigma_1)\|df_k\|^2_{L^2} = (1-\sigma_1) \la
$$
where $\sigma_1=C_{\M,k}(\rho+\ep/\rho)$.
Therefore
$$
 Q(u_0)+Q(u_-) = Q(u)-Q(u_+) \ge \la-\sigma_2
$$
where $\sigma_2=C_{\M,k}\beta^{-1}(\rho+\ep/\rho)$.
On the other hand,
$$
 Q(u_0) \le (\la+\alpha) \|u_0\|^2
$$
and
$$
 Q(u_-) \le (\la-\gamma) \|u_-\|^2,
$$
hence
$$
 \la-\sigma_2 \le Q(u_0)+Q(u_-)
 \le \la (\|u_0\|^2+\|u_-\|^2) + \alpha\|u_0\|^2 - \gamma \|u_-\|^2 .
$$
Observe that
$$
\|u_0\|^2
\le \|u_0\|^2+\|u_-\|^2
\le \|u\|^2 = \|Pf_k\|^2 \le 1+\sigma_3
$$
for $\sigma_3=C_{\M,k}(\rho+\ep/\rho)$,
where the last inequality follows from Lemma~\ref{l:P-properties}(1).
Thus
$$
 \la-\sigma_2 \le \la(1+\sigma_3) + \alpha(1+\sigma_3) -  \gamma\|u_-\|^2 ,
$$
or, equivalently
$$
 \|u_-\|^2 \le \gamma^{-1}(\sigma_2+\la\sigma_3) + \alpha\gamma^{-1}(1+\sigma_3) .
$$
The right-hand side is bounded by 
$C_{\M,k}\gamma^{-1}\beta^{-1}(\rho+\ep/\rho)+C\alpha\gamma^{-1}$.
This and \eqref{e:u+} yield the first assertion of the lemma.

\textit{2.}
The proof of the second assertion is similar, with the roles of $M$ and $X$
interchanged.
\end{proof}

\begin{theorem}
\label{t:eigenspace}
1. Let $\la=\la_k(M)$ and let $f_k$ be a corresponding unit-norm
eigenfunction of $-\Delta_M$. Then for every $\gamma\in(0,1)$,
$$
\| Pf_k - \mathbb P_{(\la-\gamma,\la+\gamma)}Pf_k\|^2
 \le C_{\M,k}\gamma^{-2}(\rho+\ep/\rho)^{1/2}
$$
provided that $\rho+\ep/\rho<C_{\M,k}^{-1}$.

2. Let $\la=\la_k(\Gamma)$ and let $u_k$ be a corresponding unit-norm
eigenfunction of $-\Delta_\Gamma$. Then for every $\gamma\in(0,1)$,
$$
\| Iu_k - \mathbb P_{(\la-\gamma,\la+\gamma)}Iu_k\|^2_{L^2}
 \le C_{\M,k}\gamma^{-2}(\rho+\ep/\rho)^{1/2}
$$
provided that $\rho+\ep/\rho<C_{\M,k}^{-1}$.
\end{theorem}

\begin{proof}
Plug $\alpha=\beta=(\rho+\ep/\rho)^{1/2}\gamma$ into Lemma \ref{l:eigenspace}.
Since the interval $(\la+\alpha,\la+\beta)$ is empty, the assumption
about eigenvalues is satisfied automatically.
The desired estimates follows from Lemma \ref{l:eigenspace}
and the relations $\alpha<\gamma$, $\alpha\gamma^{-1}=(\rho+\ep/\rho)^{1/2}$
and $\beta^{-1}=\gamma^{-1}(\rho+\ep/\rho)^{-1/2}$.
\end{proof}

The next theorem provides somewhat sharper estimates
(which are however not uniform over $\M$)
in terms of spectral gaps.

\begin{theorem}
\label{t:eigenfunction}
Let $\la$ be an eigenvalue of $-\Delta_M$ of multiplicity $m$,
more precisely,
$$
 \la_{k-1}<\la_k=\la=\la_{k+m-1} < \la_{k+m} .
$$
where $\la_j=\la_j(M)$.
Let $\delta_\la=\min\{1,\la_k-\la_{k-1}, \la_{k+m}-\la_{k+m-1}\}$
and assume that $\rho+\ep/\rho<C_{\M,k}^{-1}\delta_\la$.
Let $u_k,\dots,u_{k+m-1}$ be orthonormal eigenvectors of $-\Delta_\Gamma$
corresponding to eigenvalues $\la_k(\Gamma),\dots,\la_{k+m-1}(\Gamma)$.

Then there exist orthonormal eigenfunctions $g_k,\dots,g_{k+m-1}$ of $-\Delta_M$
corresponding to the eigenvalue $\la$ and such that
\be\label{e:uj-Pgj}
 \|u_j - Pg_j\|^2 \le C_{\M,k}\delta_\la^{-2}(\rho+\ep/\rho)
\ee
and
\be\label{e:gj-Iuj}
 \|g_j - Iu_j\|^2_{L^2} \le C_{\M,k}\delta_\la^{-2}(\rho+\ep/\rho)
\ee
for all $j=k,\dots,k+m-1$.
\end{theorem}

\begin{proof}
By Theorem~\ref{main}, the constant $C_{\M,k}$ in the bound for $\rho+\ep/\rho$
can be chosen so that $|\la_j(\Gamma)-\la_j(M)|<\frac14\delta_\la$ for all $j\le k+m$.
For every $j=k,\dots,k+m-1$, apply the second part of Lemma~\ref{l:eigenspace}
with $j$ in place of $k$, $\la'=\la_j(\Gamma)$ in place of $\la$,
$\alpha=2|\la'-\la|$ and $\beta=\gamma=\frac12\delta_\la$.
We have
$$
\la-\delta_\la<\la'-\gamma<\la<\la'+\alpha<\la'+\beta<\la+\delta_\la ,
$$
therefore the assumptions of Lemma~\ref{l:eigenspace} are satisfied
and
$$
H_{(\la'-\gamma,\la'+\alpha]}(M)=H_{\{\la\}}(M)
=\operatorname{span}\{f_k,\dots,f_{k+m-1}\} .
$$
Thus Lemma~\ref{l:eigenspace} yields that
\be\label{e:eigenfunction1}
 \|Iu_j - \tilde g_j\|^2_{L^2}
 \le C|\la'-\la|\delta_\gamma^{-1} +C_{\M,k}\delta_\gamma^{-2}(\rho+\ep/\rho)
 \le C_{\M,k}\delta_\gamma^{-2}(\rho+\ep/\rho)
\ee
where $\tilde g_j=\mathbb P_{\{\la\}}Iu_j$.
Here the second inequality follows from the fact that
$|\la'-\la|<C_{\M,k}(\rho+\ep/\rho)$ by Theorem~\ref{main}.

By Lemma \ref{l:interpolation}(1), $I$ is almost isometric
(up to an error term $C_{\M,k}\rho$)
on the linear span of $u_k,\dots,u_{k+m-1}$.
This and \eqref{e:eigenfunction1} imply that
the functions $\tilde g_k,\dots,\tilde g_{k+m-1}$
are almost orthonormal up to $C_{\M,k}\delta_\gamma^{-2}(\rho+\ep/\rho)$.
Let $\{g_j\}_{j=k}^{k+m-1}$ be the Gram--Schmidt orthogonalization
of $\{\tilde g_j\}_{j=k}^{k+m-1}$, then
$$
\|g_j-\tilde g_j\|_{L^2} \le C_{\M,k}\delta_\gamma^{-2}(\rho+\ep/\rho) .
$$ 
This and \eqref{e:eigenfunction1} imply \eqref{e:gj-Iuj}.
Now \eqref{e:uj-Pgj} follows from \eqref{e:gj-Iuj}
and Lemma \ref{l:inverse}.
\end{proof}

Note that the functions $g_k,\dots,g_{k+m-1}$ constructed in
Theorem \ref{t:eigenfunction} depend on~$\Gamma$
in rather unpredictable way. The theorem only implies that
the subspace generated by $Iu_k,\dots,Iu_{k+m-1}$
converges to $H_{\{\la\}}(M)$ as $\rho+\ep/\rho\to 0$.
A fixed basis $f_k,\dots,f_{k+m-1}$ of  $H_{\{\la\}}(M)$
is approximated by vectors $Iu_k,\dots,Iu_{k+m-1}$
transformed by an $m\times m$ orthogonal matrix.

In the case of multiplicity $m=1$,
the eigenfunction $g_k$ is unique (up to a sign) and therefore
Theorem \ref{t:eigenfunction} implies Theorem \ref{eigenfunction_1}.

\section{Volume approximation}
\label{sec:volume}

This section supplements the main results of the paper.
Here we consider various aspects of 
volume approximation in the sense Definition~\ref{d:vol-approx}.

\begin{lemma}[Marriage lemma for measures]
\label{l:marriage}
Let $X\subset M$ be a finite set.
A measure $\mu$ on $X$ is an $\ep$-approximation for $\vol$ 
(in the sense of Definition \ref{d:vol-approx}) if and only if $\vol(M)=\mu(X)$
and for every $Y\subset X$ one has $\mu(Y)\le \vol(U_\ep(Y))$.
By $U_\ep$ we denote the $\ep$-neighborhood of a set.
\end{lemma}

\begin{proof}
The proof is similar to that of Hall's Lemma for bipartite graphs.
The ``only if'' implication trivially follows from the definition.

We prove the ``if'' part by induction in $N=|X|$.
To carry on the induction, we are proving a more general
assertion where $M$ is not necessarily a manifold but
a metric measure space where the volume of a ball is positive and depends
continuously on its radius.
(In particular, this implies that every sphere has zero measure.)
Note that the requirement of Definition \ref{d:vol-approx}
that $X$ is an $\ep$-net 
follows from the assumption $\mu(Y)\le \vol(U_\ep(Y))$ applied to $Y=X$
and the fact that $\vol(M)=\mu(X)$.

Let $X=\{x_i\}_{i=1}^N$. The case $N=1$ is trivial.
Suppose that $N>1$ and the assertion holds for
every metric measure space $M'$ (with the above property)
and every subset $X'\subset M'$ of cardinality less than~$N$.
We construct a family $\{V_i(t)\}_{i=1}^N$, $t\in[0,T]$,
of coverings of $M$ by measurable sets $V_i(t)$ such that
\begin{enumerate}
 \item $V_i(0)=B_\ep(x_i)$, and $V_i(t)\subset B_\ep(x_i)$ for all $t$;
 \item the sets $V_i(T)$ are disjoint;
 \item For any set $I\subset\{1,\dots,N\}$, the volume
of the set $\bigcup_{i\in I} V_i(t)$ depends continuously on~$t$.
\end{enumerate}
Informally, to construct this family, we continuously remove from
each $V_i$ some pieces of $V_j$, $j>i$. Formally, we set $T=N$ and
sequentially define
the family for $t\in[0,1]$, $t\in[1,2]$, \dots, $t\in[N-1,N]$,
in such a way that only $V_i(t)$ changes on the interval $[i-1,i]$.
Assuming that the family is already defined for $t=i-1$, we set
$$
 V_i(i-t') = V_i^0 \cup B_{\ep t'}(x_i)
$$
for all $t'\in[0,1]$, where $V_i^0$ is the set of points in the set $V_i(i-1)$
that do not belong to any of the sets $V_j(i-1)$, $j\ne i$.

If $\vol(V_i(T))=\mu_i$ for all $i$, then the sets $V_i=V_i(T)$
satisfy Definition \ref{d:vol-approx}. Otherwise consider
a maximal interval $[0,t_0]$ such that
\be\label{e:marriage1}
 \vol \left(\bigcup\nolimits_{i\in I} V_i(t)\right) \ge \sum\nolimits_{i\in I} \mu_i
\ee
for every set $I\subset\{1,\dots,N\}$.
By continuity, such a $t_0$ exists and the inequality \eqref{e:marriage1}
turns into equality for $t=t_0$ and some set $I=I_0\subsetneq\{1,\dots,N\}$.
(Note that \eqref{e:marriage1} is always satisfied for $I=\{1,\dots,N\}$
since the sets $V_i(t)$ cover $M$ for every~$t$.)

Let $M'=\bigcup_{i\in I_0} V_i(t_0)$ and $M''=M\setminus M'$.
We apply the induction hypothesis to the spaces $M'$ and $M''$
with respective sets $X'=\{x_i\}_{i\in I_0}$ and $X''=\{x_i\}_{i\notin I_0}$,
equipped with the restrictions of $\vol$ and $\mu$.
For $M'$ and $X'$,
the assumption that $\vol(M'\cap U_\ep(Y))\ge \mu(Y)$ for all $Y\subset X'$
trivially follows from \eqref{e:marriage1}.
For $M''$ and $X''$, we verify the assumption by contradiction.
Suppose that
$$
 \vol ( M'' \cap  U_\ep(Y) ) < \mu(Y)
$$
for some $Y\subset X''$.
Let $Y=\{x_i\}_{i\in J}$ where $J\subset \{1,\dots,N\}\setminus I_0$.
Consider the set $I=J\cup I_0$. For this set we have
$$
  \vol \left(\bigcup\nolimits_{i\in I} V_i(t_0)\right)
  \le \vol(M') + \vol ( M'' \cap  U_\ep(Y) ) < \vol(M') + \mu(Y) .
$$
By the choice of $t_0$ and $I_0$, we have
$$
\vol(M') = \vol\left(\bigcup\nolimits_{i\in I_0} V_i(t_0)\right)
=\sum\nolimits_{i\in I_0} \mu_i
$$
and $\mu(Y)=\sum_{i\in J} \mu_i$ by definition.
Thus
$$
\vol \left(\bigcup\nolimits_{i\in I} V_i(t_0)\right)< \vol(M') + \mu(Y)
=\sum\nolimits_{i\in I} \mu_i ,
$$
contrary to \eqref{e:marriage1}.
This contradiction proves that $M''$ and $X''$
satisfy the induction hypothesis.
Thus $X'$ and $X''$ (equipped with the restrictions of $\mu$)
$\ep$-approximate $M'$ and $M''$ (equipped with the restriction of $\vol$).
Hence $(X,\mu)$ $\ep$-approximates $(M,\vol)$.
\end{proof}

Recall that the \textit{Prokhorov distance} \cite{Prokhorov}
$\pi(\mu,\nu)$ between two finite Borel measures $\mu$ and $\nu$ on $M$
is the infimum of all $r>0$ such that
$$
 \mu(A)\le \nu(U_r(A))+r \qquad\text{and}\qquad \nu(A)\le \mu(U_r(A))+r
$$
for every measurable set $A\subset M$.
It is well-known that weak convergence of measures
is equivalent to convergence with respect to the Prokhorov distance.
Let us introduce a similar distance $\pi_0(\mu,\nu)$, which makes sense only
if $\mu(M)=\nu(M)$. It is defined as follows: $\pi_0(\mu,\nu)$ is the infimum of all $r>0$ such that
$$
 \mu(A)\le \nu(U_r(A)) \qquad\text{and}\qquad \nu(A)\le \mu(U_r(A))
$$
for every measurable set $A\subset M$.
Clearly $\pi(\mu,\nu)\le \pi_0(\mu,\nu)$. 

Unlike Prokhorov distance, $\pi_0$ is hardly useful for
general metric measure spaces. However, in our situation $M$ is a Riemannian
manifold and one the measures is its Riemannian volume $\vol$. In this case it is
more convenient to work with $\pi_0$ and it defines the same notion of convergence to~$\vol$. 
Indeed, the following lemma holds.

\begin{lemma}
\label{l:prokhorov}
Let $\mu$ be a Borel measure on $M$ such that $\mu(M)=\vol(M)$.
Then
$$
 \pi(\mu,\vol)\le \pi_0(\mu,\vol) \le C_\M \pi(\mu,\vol)^{1/n} .
$$
\end{lemma}

\begin{proof}
As already mentioned, the first inequality
trivially follows from the definitions.
To prove the second one, let $r>0$ be such that $\pi(\mu,\vol)<r$.
We are to prove that, for a suitable $r_1=C_\M r^{1/n}>r$ and every
measurable $A\subset M$, one has $\vol(U_{r_1}(A))\ge \mu(A)$
and $\mu(U_{r_1}(A))\ge \vol(A)$. If $U_{r_1}(A)=M$, these
inequalities follow from the assumption that $\mu(M)=\vol(M)$.
Suppose that $U_{r_1}(A)\ne M$ and choose a point
$p\in M\setminus U_{r_1}(A)$. Let $q$ be a point nearest to $p$
in the closure of $A$. 
Connect $p$ to $q$ by a minimizing geodesic and let $p_1$ be
a point on this geodesic such that $|p_1q|=(r_1+r)/2$.
The triangle inequality easily implies that the ball $B_{r_2}(p_1)$
of radius $r_2=(r_1-r)/2$ is contained in the set $U_{r_1}(A)\setminus U_r(A)$.
Therefore
$$
 \vol(U_{r_1}(A)) \ge \vol(U_r(A)) + \vol(B_{r_2}(p_1)) .
$$
Since $r>\pi(\mu,\vol)$, we have $\vol(U_r(A))\ge \mu(A)-r$.
Assuming that $r$ is sufficiently small, we have
$\vol(B_{r_2}(p_1))\sim \nu_n r_2^n > r$ since $r_2>C_\M r^{1/n}$.
Therefore $\vol(U_{r_1}(A))\ge \mu(A)$.
To prove that $\mu(U_{r_1}(A))\ge \vol(A)$, apply the same argument
to the set $M\setminus U_{r_1}(A)$ in place of $A$.
This yields that
$$
\mu(M\setminus U_{r_1}(A)) \le \vol(U_{r_1}(M\setminus U_{r_1}(A))) \le \vol(M\setminus A) ,
$$
where the second inequality follows from the fact that
$U_{r_1}(M\setminus U_{r_1}(A))\subset {M\setminus A}$.
Since $\mu(M)=\vol(M)$, this implies that
$\mu(U_{r_1}(A))\ge \vol(A)$.
Hence $\pi_0(\mu,\vol)\le r_1$ and the lemma follows.
\end{proof}

\begin{proposition}
\label{p:approx-prokhorov}
Let $X\subset M$ be a finite set and $\mu$ a measure supported on $X$
with $\mu(M)=\vol(M)$. Then
$(X,\mu)$ $\ep$-approximates $(M,\vol)$ if and only if $\pi_0(\mu,\vol)\le\ep$.
\end{proposition}

\begin{remark}
\label{r:prokhorov}
This proposition together with Lemma \ref{l:prokhorov} implies that if 
$(X,\mu)$ $\ep$-approximates $(M,\vol)$ then $\pi(\mu,\vol)\le\ep$
and, conversely, if $\pi(\mu,\vol)\le c\ep^n$ 
(where $c=c(\M)$) then
$(X,\mu)$ $\ep$-approximates $(M,\vol)$.
\end{remark}

\begin{proof}[Proof of Proposition \ref{p:approx-prokhorov}]
First assume that $\pi_0(\mu,\vol)\le\ep$. Then, by the definition
of $\pi_0$, we have $\vol(U_\ep(Y))\ge \mu(Y)$ for every $Y\subset X$.
(Here we use the fact that the boundary of $U_\ep(Y)$
has zero volume.)
This and Lemma \ref{l:marriage} imply that $(X,\mu)$ $\ep$-approximates $(M,\vol)$.

Now assume that $(X,\mu)$ $\ep$-approximates $(M,\vol)$ and let
$A\subset M$ be a measurable set.
It suffices to prove that 
\be\label{e:approx-prokhorov1}
\vol(U_\ep(A))\ge \mu(A)
\ee
and
\be\label{e:approx-prokhorov2}
\mu(U_\ep(A))\ge \vol(A) .
\ee
To prove \eqref{e:approx-prokhorov1}, observe that
$$
 \vol(U_\ep(A))\ge \vol(U_\ep(A\cap X)) \ge \mu(A\cap X)=\mu(A)
$$
where the second inequality follows from Lemma \ref{l:marriage}.
To prove \eqref{e:approx-prokhorov2}, consider the set $Y=X\cap U_\ep(A)$.
By Lemma \ref{l:marriage} we have
$$
 \mu(X\setminus Y) \le \vol(U_\ep(X\setminus Y)) \le \vol(M\setminus A)
$$
where the second inequality follows from
the fact that $U_\ep(X\setminus Y)\subset M\setminus A$.
Since $\mu(X)=\vol(M)$, this implies that
$\mu(Y)\ge \vol(A)$ and \eqref{e:approx-prokhorov2} follows.
\end{proof}

\subsection*{Computing the weights}
We conclude this section by discussing how weights $\mu_i$
can be computed, for a given an $\ep$-net $X=\{x_i\}\subset M$.
There are several natural ways to associate a partition $\{V_i\}$
as in Definition \ref{d:vol-approx} to the $\ep$-net.
One is to let $\{V_i\}$ be the Voronoi decomposition of $M$
determined by~$X$.
Another possibility is to define
$V_i=B_\ep(x_i)\setminus\bigcup_{j<i} B_\ep(x_j)$.
However actual computation of the weights $\mu_i=\vol(V_i)$
is these constructions may be complicated.

A more practical approach could be the following. First split $M$ into
small subsets (of diameter at most $\ep'<\ep$) whose volumes are easy to compute.
To each of these subsets, associate one of the nearby points from~$X$.
Then the weight $\mu_i$ can be defined as
the sum of volumes of the subsets associated to the point $x_i$.
These weights define an $(\ep+\ep')$-approximation of volume.

For example, let $M\subset \R^n$ be a bounded region (rather than a closed manifold).
Then $M$ (except a small neighborhood of the boundary) can be divided
into small coordinate cubes. To each cube
one could associate the point of $X$ nearest to the cube's center.
The resulting weights are roughly equal to the volumes of Voronoi regions
but are easier to compute.

\begin{remark}
\label{r:weight}
It is an interesting problem how to derive the weights
from the distance matrix of $X$ without referring to the manifold $M$.
Ideally, one wants a nice, symmetric formula for $\mu_i$ in terms of
the distance matrix. We were not able to come up with such a formula.
However there is a straightforward algorithm based on the property that
a Riemannian metric is locally almost Euclidean.

Let $r\in (C\ep,\rho)$, $Kr^2\ll 1$. Then $r$-balls
in $M$ are bi-Lipschitz close to the $r$-ball in $\R^n$.
Moreover for each point $x_i\in M$
one can construct a bi-Lipschitz almost isometry
$\phi\co B_r(x_i)\to\R^n$ from distance functions of points
$x_j\in B_r(x_i)$.
For example, a function of the form $x\mapsto d(x,x_j)^2-d(x,x_i)^2$ is close
to a linear one in geodesic normal coordinates centered at $x_i$.
Using such functions as coordinates and post-composing with a suitable
linear transformation of $\R^n$ one gets a desired almost isometric map~$\phi$.
The image $\phi(X\cap B_r(X))\subset\R^n$ is easy to compute,
and then the problem essentially reduces to computation of volumes of Voronoi regions
(or differences of balls) in~$\R^n$.
\end{remark}

\end{document}